\documentclass[12pt,reqno]{amsart}
\usepackage{amsmath,amssymb,amsfonts,amscd,latexsym,amsthm,mathrsfs,verbatim,textcomp}
\usepackage[colorlinks,linkcolor=black,citecolor=black]{hyperref}
\usepackage{cite}
\usepackage{hypbmsec}
\usepackage{enumerate}
\usepackage{bm}
\usepackage{tikz}
\usepackage{mathtools}
\usepackage{verbatim}
\usepackage{xfrac}  
\theoremstyle{plain}

\usetikzlibrary{matrix,arrows,decorations.pathmorphing}
\usepackage[margin=1in]{geometry}
\newtheorem{lemma}{Lemma}
\newtheorem{theorem}{Theorem}
\newtheorem{corollary}{Corollary}
\newtheorem{prop}{Proposition}
\theoremstyle{remark}
\newtheorem{remark}{\bf Remark}
\newtheorem*{remark*}{\bf Remark}

\newtheorem*{hypothesis*}{\bf Hypothesis}

\newtheorem*{prob*}{\bf Problem}

\newtheorem*{quest*}{\bf Question}

\renewcommand{\Im}{\operatorname{Im}}
\renewcommand{\Re}{\operatorname{Re}}

\usepackage{etoolbox}
\patchcmd{\section}{\scshape}{\bfseries}{}{}
\makeatletter
\renewcommand{\@secnumfont}{\bfseries}
\makeatother
\numberwithin{equation}{section}
\numberwithin{lemma}{section}
\numberwithin{theorem}{section}
\numberwithin{prop}{section}
\numberwithin{remark}{section}
\newcommand{\pfrac}[2]{\left(\frac{#1}{#2}\right)}

\begin{document}

\author{Alexander Dunn}
\address{The Division of Physics, Mathematics and Astronomy, 
Caltech, 1200 E California Blvd, Pasadena, CA, 91125}
\email{ajdunn2@illinois.edu}
\subjclass[2010]{Primary 11F03, 11F66, 11F68.}
\keywords{$L$-function, quadratic twist, ranks, modular forms, elliptic curves, multiple Dirichlet series}

\title{On a problem of Hoffstein and Kontorovich}

\begin{abstract}
Let $\pi$ be a cuspidal automorphic representation of
$\operatorname{GL}_2(\mathbb{A}_{\mathbb{Q}})$
and $d$ be a fundamental discriminant.
Hoffstein and Kontorovich ask for a bound
on the least $|d|$ (if it exists)
such that the central value $L(1/2, \pi \otimes \chi_d)  \neq 0$.
The bound should be given in terms of the weight, Laplace eigenvalue 
and/or level of $\pi$.

Let $f$ be a holomorphic twist-minimal newform of even weight $\ell$, odd cubefree
level $N$, and trivial nebentypus. When $\pi \cong \pi_f$ and the squarefree part of
$N$ is of appropriate size,
we conditionally improve upon level aspect results of Hoffstein and 
Kontorovich under subconvexity (with a sub-Weyl exponent)
for automorphic $L$-functions.
As a consequence we conditionally prove that 
given an elliptic curve $E/\mathbb{Q}$ of conductor $N$,  
there exists a small twist that has Mordell--Weil rank equal to zero.
\end{abstract}

\maketitle

\section{Introduction}

\subsection{Statement of results}
Hoffstein and Kontorovich \cite{HK} pose the following
\begin{prob*} \label{mainq}
Let $\pi$ be a cuspidal automorphic representation
of $\operatorname{GL}_2(\mathbb{A}_{\mathbb{Q}})$
and $d$ be a fundamental discriminant.  
If they exist, establish a bound for the least value of $|d|$ 
(relative to the data attached to $\pi$)
such that 
\begin{equation} \label{nonvanish}
L(1/2, \pi \otimes \chi_d) \neq 0.
\end{equation}
\end{prob*}

In this paper we focus on holomorphic cuspidal newforms 
on $\operatorname{GL}_2$ over $\mathbb{Q}$.
Let $N$ will be a cubefree odd integer having factorisation 
\begin{equation} \label{factorise}
N=N_0 N_1^2, \quad \mu^2(N_0)=\mu^2(N_1)=1,
\end{equation}
and $\ell$ be an even integer. Let $S_{\ell}(N)$ denote the space of holomorphic cusp forms 
on $\Gamma_0(N)$ with weight $\ell$ and trivial nebentypus. Similarly, let 
$S^{\text{new}}_{\ell}(N)$ be the space of such arithmetically normalised newforms 
(first Fourier coefficient is normalised to be $1$).
For each $f \in S^{\text{new}}_{\ell}(N)$, let $\pi_f \cong \otimes^{\prime}_p \pi_{f,p}$ 
denote the corresponding cuspidal automorphic representation of 
$\operatorname{GL}_2(\mathbb{A}_{\mathbb{Q}})$.  

Let $\mathfrak{c}(\pi_f):=\prod \mathfrak{c}(\pi_{f,p})=N$ 
denote the arithmetic conductor of $\pi_f$. 
It is said that $f \in S_{\ell}^{\text{new}}(N)$ is \emph{twist-minimal} if 
$\mathfrak{c}(\pi_f \otimes \psi) \geq \mathfrak{c}(\pi_f)$ for all Dirichlet 
characters $\psi$. 

In the case of holomorphic cusp forms,
Waldspurger's formula \cite{Wal}
and the Riemann--Roch theorem
imply that that there exists a quadratic twist of fundamental discriminant 
 $d$ satisfying $|d| \ll_{\varepsilon} (\ell N)^{1+\varepsilon}$
and \eqref{nonvanish}.
This bound is on par with what a ``convexity bound" 
for the relevant multiple Dirichlet series would imply 
(see section \ref{sketch} for more discussion). 

Hoffstein and Kontorovich 
\cite{HK} prove bounds of the above quality 
for general $\pi$ on $\operatorname{GL}_r(\mathbb{A}_{\mathbb{Q}})$, $r=1,2,3$.
The investigations in \cite{HK} are restricted to $r \leq 3$ in order to guarantee 
the relevant multiple Dirichlet series $Z(s,w;\pi)$
have meromorphic continuation past the point $(1/2,1)$. 

Conditional on a strong uniform quantitative subconvexity 
hypothesis for automorphic $L$-functions,
we improve the bound $|d| \ll_{\varepsilon} (\ell N)^{1+\varepsilon}$ 
in the level aspect, for some levels $N$.

\begin{theorem} \label{mainthm}
Let $N=N_0 N_1^2$ be an odd cubefree odd integer
and $f \in S^{\emph{new}}_{\ell}(N)$ be twist minimal. Suppose there is $0<\delta_1<1/4$ 
such that for all positive fundamental discriminants $d$ we have 
\begin{equation} \label{GL2lindsub}
L(1/2+it,\pi_f \otimes \chi_{d}) \ll_{\ell,A,\varepsilon}  
\mathfrak{c}(\pi_f \otimes \chi_d)^{1/4-\delta_1+\varepsilon} (1+|t|)^{A},
\end{equation}
for some $A>0$. Suppose there exists a $0<\delta_2<1/4$ such that 
\begin{equation} \label{symsub}
L(1/2+it,\emph{Sym}^2 \pi_f)  \ll_{\ell,A,\varepsilon} \mathfrak{c}(\emph{Sym}^2 \pi_f)^{1/4-\delta_2+\varepsilon} (1+|t|)^{A},
\end{equation}
for some $A>0$.

Then there exists a fundamental discriminant $d_0$ satisfying
\begin{equation} \label{discbound}
1 \leq d_0 \ll_{\ell,\varepsilon} N^{\varepsilon} \Big( N^{9/4-7 \delta_1} N_0^{7/4-5 \delta_1}
+\frac{N^{1-6 \delta_2  }}{N_0^{2 \delta_2}}  \Big)
\quad \text{and} \quad (d_0,2N)=1, 
\end{equation}
such that
\begin{equation*}
L(1/2,\pi_f \otimes \chi_{d_0}) \neq 0.
\end{equation*}
The constant in \eqref{discbound} is ineffective in terms of $\varepsilon>0$.
\end{theorem}
\begin{remark*}
~\begin{itemize}
\item Theorem \ref{mainthm} gives a power saving improvement over \cite{HK} when $N$ has 
squarefree part 
\begin{equation*}
N_0 \ll N^{(7 \delta_1-5/4)/(7/4-5 \delta_1)+o(1)}.
\end{equation*}
Thus Theorem \ref{mainthm} requires a strong input: $5/28=0.178 \ldots<\delta_1<1/4$ and $\delta_2>0$.
\item The limiting case $\delta_1=1/4$ (i.e the generalised Lindel\"{o}f hypothesis)
gives a power saving improvement as soon as 
\begin{equation*}
1<N_0 \ll N^{1-o(1)}.
\end{equation*}
\item Michel and Venkatesh \cite{MV} solved the uniform subconvexity problem on 
$\operatorname{GL}_2(\mathbb{A}_{\mathbb{Q}})$. However, we are far away from exhibiting the 
sub-Weyl exponent needed for Theorem ~\ref{mainthm}. Subconvexity of the symmetric square is another
major open problem.
\item It is likely that an improved Heath--Brown quadratic large sieve \cite{HB}
 for modular $L$-functions would also yield similar results.
\item Existence of certain $d$ satisfying \eqref{nonvanish} can be ruled out by root number considerations.
Suppose $N$ were a perfect square in Theorem \ref{mainthm} and $\pi_f$ has root number $\varepsilon(\pi_f)=-1$.
Then $\varepsilon(\pi_f \otimes \chi_d)=\varepsilon(\pi_f) \chi_d(-N)=-1$ for all fundamental discriminants $d>0$ satisfying $(d,2N)=1$.
Then all such central values $L(1/2, \pi_f \otimes \chi_d)$ vanish.
\item Let $F$ be a number field. Problems concerning the existence of $d$ (and infinitude of such $d$) 
such that \eqref{nonvanish} holds for a fixed cuspidal automorphic representation $\pi$ of
$\operatorname{GL}_2(\mathbb{A}_F)$ was resolved by Friedberg and Hoffstein in \cite{FH}.
This was built on a host of works \cite{BFH1,BFH2,Iw,MM}. Ono and Skinner \cite{OS} 
give lower bounds on the number fundamental discriminants $0<|d| \leq X$ such that \eqref{nonvanish}
holds for a fixed holomorphic $\pi$ on $\operatorname{GL}_2(\mathbb{A}_{\mathbb{Q}})$.
\end{itemize}
\end{remark*}

Let $E/\mathbb{Q}$ be an elliptic curve with Weierstrass equation 
$y^2=x^3+ax+b$, having conductor $N$. Let $E^{(d)}/\mathbb{Q}$ 
define the twisted curve given by the equation
$dy^2=x^3+ax+b$. An immediate consequence of Theorem \ref{mainthm}, the
modularity theorem \cite{BCDT}, and a result of Kolyvagin \cite{KVA} is

\begin{corollary} \label{elliptic}
Let $N=N_0 N_1^2$ be an odd cubefree integer and
$E/\mathbb{Q}$ be a twist minimal elliptic curve of conductor $N$.
Suppose $L(s, \pi_E \otimes \chi_d)$ satisfies \eqref{GL2lindsub} for all fundamental discriminants $d$,
and that $L(s,\emph{Sym}^2 \pi_E)$ satisfies \eqref{symsub}. 
Then there exists a fundamental
discriminant $d_0$ satisfying the conditions in \eqref{discbound}
such that $E^{(d_0)}$ has Mordell--Weil rank equal to zero.
\end{corollary}

Note that primes $p \mid N_0$ (resp. $p \mid N_1$) correspond to primes of multiplicative reduction
(resp. additive reduction) for $E/\mathbb{Q}$. Thus the savings obtained in 
Corollary \ref{elliptic} are governed  
by the reduction types of bad primes.  

Impressive work of Petrow \cite{Pet} conditionally establishes (under GRH)
 the existence of a
fundamental discriminant $d_0$
satisfying 
\begin{equation*}
 0<d_0 \ll_A \frac{N \ell}{( \log N \ell )^{A}}, \quad \varepsilon(\pi_f \otimes \chi_{d_0})=-1,
  \quad \text{and} \quad (d_0,2N)=1,
 \end{equation*}
 such that the derivative $L^{\prime}(1/2,\pi_f \otimes \chi_{d_0}) \neq 0$. It is expected
 that the methods of this paper would carry over to that situation as well.

We close by mentioning that a power saving improvement on the ``convexity" bound
 $|d| \ll_{\varepsilon} (\ell N)^{1+\varepsilon}$
in the weight aspect would have interesting applications to non-vanishing of certain 
$\operatorname{GL}_3 \times \operatorname{GL}_2$ $L$-functions. For this, 
one may consult work of Liu and Young \cite{LY}.
 
 \subsection{Heuristics and outline}  \label{sketch}
 Here we outline the main ideas in this paper, ignoring most technicalities (i.e. 
the presence of smooth functions).
We remind the reader 
that $\ell$ is an arbitrary but fixed positive even integer, and that $N$ is an odd cubefree
integer allowed to move and has factorisation \eqref{factorise}.

For now we ignore the requirement that $d$ be a fundamental discriminant 
(it will be addressed in momentarily). 

In order to obtain Theorem \ref{mainthm}, 
one could try to prove an asymptotic formula
\begin{equation} \label{heuristicmom}
\sum_{d \sim X } \frac{1}{d^{1/2}} L(1/2,\pi_f \otimes \chi_d)=T_{\pi_f}(X)+E_{\pi_f}(X),
\end{equation}
where $T_{\pi_f}(X)$ and $E_{\pi_f}(X)$ are the main and error terms 
respectively. Typically $T_{\pi_f}=c_{\pi_f} \cdot X^{1/2+o(1)}$ 
where $|c_{\pi_f}| \gg_{\varepsilon} N^{-\varepsilon}$. If 
$T_{\pi_f}(X)$ dominates the error term for some range, 
then the existence of a $d \sim  X$ satisfying \eqref{nonvanish}
would immediately follow.
 
Each $L$-function on the left side of \eqref{heuristicmom}
has conductor $\asymp_{\ell} X^2 N$. Let
$1 \leq R \ll_{\ell} X^2 N$ be a paramater chosen later.
We use the unbalanced approximate functional equation \cite[Theorem~5.3]{IK}
to open each summand in \eqref{heuristicmom} (this is morally the same as applying Voronoi summation).
Interchanging the order of summation gives
\begin{align} \label{coreexp}
\sum_{d \sim X } \frac{1}{d^{1/2}} L(1/2,\pi_f \otimes \chi_d) & \approx 
\sum_{1 \leq n \ll R} \frac{\lambda_f(n)}{n^{1/2}} 
\sum_{\substack{d \sim X \\ (d,2N)=1 }} \frac{\chi_d(n)}{d^{1/2}}  \nonumber \\
&+ \varepsilon(f) \sum_{1 \leq n \ll X^2 N/R} \frac{\lambda_f(n)}{n^{1/2}} 
\sum_{\substack{d \sim X \\ (d,2N)=1}} \frac{\chi_d(-Nn)}{d^{1/2}}, 
\end{align}
where $\varepsilon(\pi_f \otimes \chi_d)=\varepsilon(f) \chi_d(-N)$
is the root number of  $L(1/2,\pi_f \otimes \chi_d)$.

Applying Poisson summation
to the first (resp. second) $d$ summation in  gives a dual sum
whose length is $|d| \ll_{\ell} R/X$ (resp. $|d| \ll_{\ell} XN^2/ R$). In order to gain from this move 
we would need $X \gg_{\ell} N$. 
This is a deadlock for our problem.

To circumvent this we use the conductor drop coming from the factorisation
 $N=N_0 N_1^2$. Observe that $\chi_d(-N)=\chi_d(-N_0 N_1^2)=\chi_d(-N_0)$,
 and repeating the above Poisson step, we only need $X \gg (N N_0)^{1/2}$
 to shorten both summations. 
 
 Post Poisson, one would expect the main terms to come from the zero frequencies 
 from both sums, so we ignore these terms for now. 
 Subconvexity of the symmetric square \eqref{symsub} ensures the error incurred 
 from the contour shifting required for main term extraction is acceptable. We point 
 out that \cite[Example~1]{Li2} gives $\mathfrak{c}(\text{Sym}^2 \pi_f)=N_0^2 N_1^3$ (as opposed
 to the full $N^2$), and
 so we also benefit from this conductor drop.

 We choose $R:=X(NN_0)^{1/2}$ to balance the lengths of both $d$ summations. 
Interchanging the orders of summation post Poisson 
and application of the hypothesis \eqref{GL2lindsub} would in 
theory yield a result. In reality, there is a conductor raising penalty incurred
by the M\"{o}bius 
inversion of the condition $(d,2N)=1$ (equivalent to $(d,2 \mathsf{rad}(N))=1$),
and this is handled in the endgame calculation.
 
Unfortunately not all integers are fundamental discriminants, and to make the above approach rigorous 
we use certain 
combinatorial weights $\mathcal{P}_d(1/2;\pi_f)$ \cite{BFH,CG1,CG2,Di} ($\mathcal{P}_d(s;\pi_f)$
is Dirichlet polynomial) coming from the 
theory of multiple Dirichlet series. We
consider the perturbed moment 
\begin{equation*}
\sum_{\substack{d=d_0 d_1^2 \sim X \\ (d,2N)=1 \\ \mu^2(d_0)=1 }}
 \mathcal{P}_d(1/2;\pi_f)  L(1/2,\pi_f \otimes \chi_{d_0}).
\end{equation*}

Mellin inversion and other standard moves bring into play a multiple Dirichlet series
$Z(s,w;\pi_f)$. For $(s,w) \in \mathbb{C}^2$ 
in an appropriate region of absolute convergence,
\begin{equation*}
Z(s,w;\pi_f) \approx \sum_{\substack{ d \geq 1\\ d=d_0 d_1^2 \\ (d,2N)=1}} \frac{\mathcal{P}_d(s;\pi_f)  
L(s,\pi_f \otimes \chi_{d_0}) }{d^w},
\end{equation*}
and for another region,
\begin{equation*}
Z(s,w;\pi_f) \approx \sum_{\substack{ n \geq 1\\ n=n_0 n_1^2 \\ (n,2N)=1}} \frac{\mathcal{Q}_n(w;\pi_f)  
L(w,\widetilde{\chi}_{n_0} \chi_N)}{n^s},
\end{equation*}
where the $\mathcal{Q}_n(w;\pi_f)$ are certain combinatorial weights.

The series $Z(s,w;\pi_f)$ has finite group of functional equations
isomorphic to the dihedral group of order $8$ (Weyl group of the root system $C_2$)
with generators
\begin{equation*}
\gamma_1(s,w)=(1-s,w+2s-1) \quad \text{and} \quad \gamma_2(s,w)=(s+w-1/2,1-w).
\end{equation*}
Consequently $Z(s,w;\pi_f)$
has full meromorphic continuation to $\mathbb{C}^2$ with well
understood polar hyperplanes.

The notion of a ``convexity" bound for $Z(s,w;\pi_f)$ is \emph{a priori} not well defined. 
It depends on what is assumed about each summand (both $L(s,\pi_f \otimes \chi_{d_0})$
and $L(w,\widetilde{\chi}_n \chi_N)$)
in various regions of absolute convergence for $Z(s,w;\pi_f)$. If one assumes the full Lindel\"{o}f 
hypothesis for each $L$-function (in the $d$ and $N$-aspects), this would give  
\begin{equation*}
Z(1/2,1/2+i t;\pi_f) \ll_{\ell,\varepsilon,A}  N^{1/2+\varepsilon} (1+|t|)^{A},
\end{equation*}
for some $A>0$ (see \cite[Proposition~3.20]{HK}). 
A Mellin inversion and contour shifting argument to the half line as 
in \cite[Remark~4.2]{HK} would yield the bound 
$|d| \ll_{\ell,\varepsilon} N^{1+\varepsilon}$ mentioned in the introduction. Thus it is natural
to ask whether one can conditionally go beyond this bound, as we have done in this paper.

A differing school of thought (\cite{Blo} and \cite{Dah}) is that the 
``convexity" bound should correspond to  
 a Lindel\"{o}f-on-average bound in the $d$-aspect, and the 
convexity bound for the other parameters, in the regions of absolute convergence.
The advantage of this regime is that these assumptions can be 
established unconditionally using Heath-Brown's quadratic large sieve
\cite{HB}.
In this weaker setting, 
Blomer \cite{Blo} established ``subconvex" bounds for 
$\operatorname{GL}_1$ multiple Dirichlet series in the $t$-aspect, and his student Dahl 
\cite{Dah} for the same series in the level aspect.

The approximate functional equation/Voronoi move in 
our heuristic is mimicked by application of the functional equation
corresponding to $\gamma_1$. Similarly, the Poisson step 
is mimicked by application of $\gamma_2$.
Blomer observed \cite{Blo} the equivalence between Poisson and $\gamma_2$, 
and we make crucial use of this 
insight in this paper.

An alternative approach could be to M\"{o}bius invert the squarefree
condition as in \cite{Pet}. However such an involved computation is not necessary
given that we are only concerned with non-vanishing, and not the full moment sieved
to fundamental discriminants.

The global descriptions of the combinatorial weights we use 
were first discovered by Bump, Friedberg and Hoffstein \cite[Theorem~1.2]{BFH} 
using brute force methods. We choose to build our weights locally using work
of Chinta and Gunnells \cite{CG1,CG2} and Diaconu \cite{Di}. One pleasing novelty 
of this approach is how the algebraic structure of these weights naturally give 
each term in the Euler product of $L(s,\text{Sym}^2 \pi_f)$ (see Lemma \ref{R1}).
The main term of the first moment involves the constant $L(1,\text{Sym}^2 \pi_f)$.

The methods in this paper could probably be extended to the 
$\operatorname{GL}_r(\mathbb{A}_{\mathbb{Q}})$ cases for $r=1,3$.

Section \ref{preliminaries} contains basic $L$-functions facts. Section \ref{constrmds}
records the relevant multiple Dirichlet series we use and a careful derivation of the scattering 
matrices that appear in the functional equations for them. Section \ref{mainargument}
makes the Voronoi and Poisson heuristic above rigorous and Section \ref{endgame} 
contains the proof of Theorem \ref{mainthm}.

This work is intended be a pleasant interaction between the multiple Dirichlet series 
and approximate functional equation perspectives, and the equivalences between them.

\section*{Acknowledgements}
I warmly thank Jeffrey Hoffstein for introducing me to his open question 
and for many fruitful discussions. I am also grateful to Maksym Radziwi\l{}\l{} and Matthew Young
for their helpful comments, and Henri Darmon for an insight regarding root numbers
(communicated via Maksym Radziwi\l{}\l{}).  I would also like to thank Alex Kontorovich 
for his helpful feedback on my manuscript.

\section*{Conventions}
Unless otherwise stated, the implied constants are allowed to depend on $\ell$
and an arbitrarily small constant $\varepsilon>0$. The quantity $\varepsilon>0$
may differ in each instance it appears.

\section{Preliminaries} \label{preliminaries}
\subsection{$L$-functions}
Let the Fourier expansion of $f \in S_{\ell}^{\text{new}}(N)$
be given by 
\begin{equation} \label{fourier}
f(\tau):=\sum_{n=1}^{\infty} \lambda_f(n) n^{(\ell-1)/2} e^{2 \pi i n \tau}  
\in S^{\text{new}}_{\ell}(N), \quad \lambda_f(1)=1, \quad \tau \in \mathbb{H}.
\end{equation}

The $L$-function attached to $\pi_f$ is 
\begin{equation} \label{Lfunction}
L(s,\pi_f):=\sum_{n=1}^{\infty} \frac{\lambda_f(n)}{n^s}=\prod_{p \nmid N} 
\Big(1-\frac{\lambda_f(p)}{p^{s}}+\frac{1}{p^{2s}} \Big)^{-1}  \prod_{p \mid N}
\Big(1-\frac{\lambda_f(p)}{p^{s}} \Big)^{-1}, \quad \Re s>1.
\end{equation}
Each Euler factor on the far right side of \eqref{Lfunction} is denoted $L(s,\pi_{f,p})$.
The $L$-function $L(s,\pi_{f})$ has analytic continuation to all 
of $\mathbb{C}$, and satisfies the functional equation 
\cite[Theorem~14.17]{IK}
\begin{equation*}
\Lambda(s,\pi_{f})=\varepsilon(\pi_f) \Lambda(1-s,\widetilde{\pi}_{f}),
\end{equation*}
where $\varepsilon(\pi_f)$ is the root number ($|\varepsilon(\pi_f)|=1$), $\widetilde{\pi}_{f}$ denotes the 
contragredient, and 
\begin{equation*}
\Lambda(s,\pi_f):=\mathfrak{c}(\pi_f)^{s/2} \pi^{-s} \Gamma \Big( \frac{s+\frac{\ell-1}{2}}{2} \Big) 
\Gamma \Big( \frac{s+\frac{\ell+1}{2}}{2} \Big) L(s,\pi_f).
\end{equation*}
Since $f$ has trivial nebentypus, $\pi_f \cong \widetilde{\pi}_f$ ($\pi_f$ is self-contragredient).

Let $\chi$ be a primitive Dirichlet character with conductor $Q$, and 
\begin{align*}
L(s,\pi_f \otimes \chi)&:=\sum_{n=1}^{\infty} \frac{\lambda_f(n) \chi(n)}{n^s} \\
&=\prod_{p \nmid N} 
\Big(1-\frac{\lambda_f(p) \chi(p)}{p^{s}}+\frac{\chi^2(p)}{p^{2s}} \Big)^{-1}  \prod_{p \mid N}
\Big(1-\frac{\lambda_f(p) \chi(p)}{p^{s}} \Big)^{-1}, \quad \Re s>1.
\end{align*}
If $(N,Q)=1$, then by \cite[Proposition~14.20]{IK} we have
\begin{equation} \label{functionaleq2}
\Lambda(s,\pi_{f} \otimes \chi)=\varepsilon(\pi_f \otimes \chi) \Lambda(1-s,\widetilde{\pi_{f} \otimes \chi}),
\end{equation}
with $\mathfrak{c}(\pi_f \otimes \chi)=NQ^2$ and root number
\begin{equation} \label{twistroot1}
\varepsilon(\pi_f \otimes \chi)=\varepsilon(\pi_f) \chi(N) g^2_{\chi},
\end{equation}
where $g_{\chi}$ is the normalised Gauss sum attached to $\chi$.  
A more convenient formula is 
\begin{equation} \label{twistroot2}
\varepsilon(\pi_f \otimes \chi_d)=\varepsilon(\pi_f) \chi_d(-N).
\end{equation}
If $\varepsilon(\pi_f \otimes \chi_d)=-1$ then $L(1/2,\pi_f \otimes \chi_d)=0$.

\subsection{Dirichlet $L$-functions}
Let $\chi$ be a character modulo $Q$.
The Dirichlet $L$-function 
\begin{equation*}
L(w,\chi):=\sum_{n=1}^{\infty} \frac{\chi(n)}{n^{w}}, \quad \Re(w)>1,
\end{equation*}
has meromorphic continuation to all of $\mathbb{C}$. It has simple pole at $w=1$ when 
$\chi=\mathbf{1}$ is the principal character modulo $Q$ and is entire if $\chi \neq \mathbf{1}$. 

It $\chi$ is primitive modulo $Q$ then 
we have the functional equation 
\begin{equation} \label{dirtwistfunc}
L(w,\chi)= \frac{g_{\chi}}{i^{\mathfrak{a}}} 
\Big(\frac{Q}{\pi} \Big)^{1/2-w}
\frac{\Gamma \big(\frac{1-w+\mathfrak{a}}{2} \big)}{\Gamma \big( \frac{w+\mathfrak{a}}{2} \big)} 
L(1-w,\overline{\chi}), \quad w \in \mathbb{C},
\end{equation}
where 
\begin{equation*}
\mathfrak{a}=\begin{cases}
0, & \quad \chi(-1)=1 \\
1, &  \quad \chi(-1)=-1.
\end{cases} 
\end{equation*}

\subsection{Real characters}
Following \cite{DGH} and \cite{Blo} we will use a slightly different notation for 
real characters in the remainder of the paper. 

Let $d$ and $n$ be odd positive integers with factorisations
\begin{equation} \label{nd}
d=d_0 d_1^2 \quad \text{and} \quad n=n_0 n_1^2, \quad \text{where} \quad  \mu^2(d_0)=\mu^2(n_0)=1.
\end{equation}
Write 
 \begin{equation*}
 \chi_d(n):=\Big( \frac{d}{n}  \Big)=\widetilde{\chi}_n(d).
 \end{equation*}
 The character $\chi_d$ is the Jacobi--Kronecker symbol of conductor $d_0$ if
 $d \equiv 1 \pmod{4}$ and $4d_0$ if $d \equiv 3 \pmod{4}$. We have $\chi_{d}(-1)=1$, so $\chi_d$ is even.
  By quadratic reciprocity 
 we have 
 \begin{equation} \label{quadrep}
 \widetilde{\chi}_n=\begin{cases}
 \chi_n, & \quad n \equiv 1 \pmod{4} \\
 \chi_{-n}, & \quad n \equiv 3 \pmod{4}.
 \end{cases}
 \end{equation}
  
 For $a \in \{\pm 1, \pm 2\}$, let $\chi_a$ denote the four characters modulo $8$.
That is,
$\chi_1$ is the trivial character, $\chi_{-1}$ is induced from the 
non-trivial character modulo $4$,
$\chi_2=1$ if and only if $n \equiv 1,7 \pmod{8}$ and $\chi_{-2}(n)=1$ if and only if 
$n \equiv 1,3 \pmod{8}$. 

All real primitive characters can be constructed from products $\chi_{d_0} \chi_{a}$
with $d_0$ odd and squarefree and $a \in \{ \pm 1, \pm 2 \}$.

In the body of the paper we will also 
require the $d$ from \eqref{nd} to satisfy $(d,2N)=1$. 
Let 
\begin{equation*}
\text{Div}(N):=\{a \cdot c : a \in \{\pm 1, \pm 2 \}  \text{ and } c \mid \mathsf{rad}(N) \},
\end{equation*}
where $\mathsf{rad}(m):=\prod_{p \mid m} p$ 
denotes the usual radical of an integer $m$.

When working with multiple Dirichlet series,
we write primitive real Dirichlet characters using 
the convention
\begin{equation} \label{dircond}
\chi_{d_0} \chi_{ac} \quad \text{where}  \quad (d_0,2N)=1, \quad \text{and} \quad
ac \in \text{Div}(N).
\end{equation}

\subsection{Subconvexity hypotheses}
Let $N$ be as in \eqref{factorise}, $f \in S^{\text{new}}_{\ell}(N)$ be twist minimal,  
$a c \in \text{Dic}(N)$, and $d_0$ squarefree with $(d_0,2N)=1$.
Then we have 
$\mathfrak{c}(\pi_f \otimes \chi_{d_0} \chi_{ac}) \mid 8 Nc d_0^2$ by \cite[Proposition~3.1]{AL},.
Thus we can relax the subconvexity hypothesis \eqref{GL2lindsub} becomes
\begin{equation} \label{GL2lindelof}
L(1/2+it,\pi_f \otimes \chi_{d_0} \chi_{ac}) \ll_{\ell,A,\varepsilon}  (Ncd_0^2)^{1/4-\delta_1+\varepsilon} (1+|t|)^{A}.
\end{equation}

We have $\mathfrak{c}(\text{Sym}^2 \pi_f)=N_0^2 N_1^3$ by \cite[Example~1]{Li2}, 
and so the
subconvexity hypothesis \eqref{symsub} becomes
\begin{equation} \label{GL3lindelof}
L(1/2+it,\text{Sym}^2 \pi_f)  \ll_{\ell,A,\varepsilon} (N_0^2 N_1^3)^{1/4-\delta_2+\varepsilon} (1+|t|)^{A},
\end{equation}
for some $A>0$. 

A consequence of \eqref{GL2lindelof}, dyadic partition of unity, Mellin inversion and a 
trivial estimation of Euler factors attached to primes $p \mid 2N$
is the following
\begin{lemma} \label{lindelof}
Suppose $f \in S^{\emph{new}}_{\ell}(N)$ satisfies \eqref{GL2lindelof} and has
Fourier expansion \eqref{fourier}.
Then for $ac \in \emph{Div}(N)$, $d_0$ squarefree with $(d_0,2N)=1$,
 $Q \geq 1$, $t \in \mathbb{R}$ and $\varepsilon>0$ we have 
\begin{equation*}
\bigg| \sum_{\substack{1 \leq n \leq Q \\ (n,2N)=1}} 
\frac{\lambda_f(n) \chi_{a c}(n) \chi_{d_0}(n)}{n^{1/2+it}} \bigg | 
 \ll_{\ell,A,\varepsilon}   Q^{\varepsilon} (Ncd_0^2)^{1/4-\delta_1+\varepsilon} (1+|t|)^{A},
\end{equation*}
for some $A>0$.
\end{lemma}

\section{Construction of the Multiple Dirichlet Series} \label{constrmds}

The classical approach to multiple Dircihlet series and automorphic forms is well summarised by 
\cite{BFH,BBCFH}. 

Global formulas for the combinatorial weights we need were originally discovered by 
Bump, Friedberg and Hoffstein \cite[Theorem~1.2]{BFH} using brute force methods.
We opt to describe the multiple Dirichlet series locally using the recipe briefly 
described on \cite[pg.~331]{CG1}. Such a description makes clear 
the relationship between the local representations $\pi_{f,p}$ and
the $p$-\emph{part} of the multiple Dirichlet series relevant to our situation.

\subsection{Chinta--Gunnells action in the case $A_3$}
The power of the Chinta--Gunnells construction \cite{CG1,CG2} is that it works for arbitrary simply laced root systems.
For simplicity we focus on the Dynkin diagram $A_3$, which is sufficient for our purposes. A more general summary 
can be found in \cite[Section~2.1]{Di}.

Note that the multiple Dirichlet series we use in Section \ref{mdsdef}
is a two variable specialisation of a three variable series whose group of functional
equations is isomorphic to the Weyl group of $A_3$, denoted $W(A_3)$. Hence the specialised series 
has group of functional equations isomorphic to $W(C_2)$ 
($C_2$ is the Dynkin fold of $A_3$), as mentioned in section \ref{sketch}.

After fixing an ordering of the roots, let $A_3=A_3^{+} \cup A_3^{-}$ denote a decompositon
into positive and negative roots. Let  $\alpha_1,\alpha_2$ and $\alpha_3$ be simple roots 
(where $\alpha_3$ corresponds to the central node of the Dynkin diagram of $A_3$)
and $\sigma_i \in W(A_3)$ be the simple reflection through a hyperplane perpendicular 
to $\alpha_i$. The simple reflections $\sigma_i$ generate the Weyl group and satisfy  
\begin{equation*}
(\sigma_i \sigma_j)^{r_{ij}}=1, \quad \text{with} \quad r_{ii}=1, 
\quad r_{12}=r_{21}=2 \quad \text{and} \quad r_{13}=r_{31}=r_{23}=r_{32}=3.
\end{equation*}
The action of simple reflections on 
roots is given by 
\begin{equation*}
\sigma_i \alpha_j=
\begin{cases}
\alpha_i+\alpha_j ,& \text{ if } \alpha_i \text{ and } \alpha_j \text{ are adjacent} \\
-\alpha_j, & \text{ if } i=j \\
\alpha_j, & \text{ otherwise}.
\end{cases}
\end{equation*}

Let $\Lambda(A_3)$ be the root lattice of $A_3$. Each element $\lambda \in \Lambda(A_3)$
has a unique representation as an integral combination of simple roots 
\begin{equation}
\lambda=k_1 \alpha_1+k_2 \alpha_2+k_3 \alpha_3.
\end{equation}
Set $\boldsymbol{z}:=(z_1,z_2,z_3)$ and for $\lambda \in \Lambda(A_3)$,
set $\boldsymbol{z}^{\lambda}:=z_1^{k_1} z_2^{k_2} z_3^{k_3}$. 
Fix a parameter $q \geq 1$ (we will later take $q=p$ prime).
Define  
${}^{\epsilon_i} \boldsymbol{z}=\boldsymbol{z}^{\prime}$, where 
\begin{equation*}
z_j^{\prime}=\begin{cases}
-z_{j}, & \text{ if } i \text{ and } j \text{ are adjacent} \\
z_{j}, & \text{otherwise},
\end{cases}
\end{equation*}
and ${}^{\sigma_i} \boldsymbol{z}=\widetilde{\boldsymbol{z}}$,
where
\begin{equation*}
\widetilde{z_j}=\begin{cases}
\sqrt{q} z_i z_j, & \text{ if } i \text{ and } j \text{ are adjacent} \\
1/(qz_j), & \text{ if } i=j  \\
z_j & \text{ otherwise}.
\end{cases}
\end{equation*}

For $h \in \mathbb{C}(\boldsymbol{z})$,
let 
\begin{equation} \label{plusminus}
h_i^{\pm}(\boldsymbol{z}):=\frac{1}{2} \big( h(\boldsymbol{z}) \pm h({}^{\epsilon_i} \boldsymbol{z}) \big).
\end{equation}
The action of a simple reflection $\sigma_i$ on $h \in \mathbb{C}(\boldsymbol{z})$ is given by  
\begin{equation} \label{simpleaction}
(h \lvert \sigma_i)(\boldsymbol{z}):=-\frac{1-qz_i}{qz_i(1-z_i)} h_i^{+}({}^{\sigma_i} \boldsymbol{z})+
\frac{1}{\sqrt{q} z_i } h_i^{-}({}^{\sigma_i} \boldsymbol{z}).
\end{equation}
This extends to a $W(A_3)$-action on $\mathbb{C}(\boldsymbol{z})$ \cite[Lemma~3.2]{CG1}.

Using the Weyl group action in \eqref{simpleaction}, Chinta and Gunnells \cite{CG1,CG2} constructed
a $W(A_3)$ invariant function 
$g_{A_3}(\boldsymbol{z}) \in \mathbb{C}(\boldsymbol{z})$ such that
\begin{itemize}
\item $g_{A_3}(\boldsymbol{0};q)=1$;
\item for each $i=1,2,3$, the function $(1-z_i) \cdot g_{A_3}(\boldsymbol{z};q) \lvert_{z_j=0 \text{ for all } j \text{ adjacent to } i}$
is independent of $z_i$.
\end{itemize}
The rational function satisfying the above conditions is unique \cite{W2,W1}. 

A straightforward computation verifies that
\begin{equation}
g_{A_3}(\boldsymbol{z};q):=\frac{1-z_1 z_3-z_2 z_3 +z_1 z_2 z_3+q z_1 z_2 z_3^2-qz_1^2 z_2 z_3^2
-q z_1 z_2^2 z_3^2+q z_1^2 z_2^2 z_3^3 }
{(1-z_1)(1-z_2)(1-z_3)(1-q z_1^2 z_3^2)(1-qz_2^2 z_3^2)(1-q^2 z_1^2 z_2^2 z_3^2)} \label{gdef}
\end{equation}
is the desired function \cite[Example~3.7]{CG1} (our vertices are labelled differently).

Expand $g_{A_3}(\boldsymbol{z}:q)$ in a power series 
\begin{equation} \label{powerser}
g_{A_3}(\boldsymbol{z};q)=\sum_{k_1,k_2,j \geq 0} a(k_1,k_2,j;q) z_1^{k_1} z_2^{k_2} z_3^{j}.
\end{equation} 
Define the symmetric polynomials $P_j(z_1,z_2;q) \in \mathbb{C}[z_1,z_2]$ by the expression 
\begin{align} \label{Pjdef}
g_{A_3}(\boldsymbol{z};q)&:=g^{+}_{1}(\boldsymbol{z};q)+g^{-}_{1}(\boldsymbol{z};q) \nonumber \\
&=(1-z_1)^{-1} (1-z_2)^{-1} \sum_{j \text{ even}} P_{j}(z_1,z_2;q) z_3^{j}
+\sum_{j \text{ odd}} P_{j}(z_1,z_2;q) z_3^{j}.
\end{align}
Similarly, polynomials $Q_{\boldsymbol{k}}(z_3;q) \in \mathbb{C}[z_3]$ can be defined by the relationship
\begin{align} \label{Qkdef}
g_{A_3}(\boldsymbol{z};q)&:=g^{+}_{3}(\boldsymbol{z};q)+g^{-}_{3}(\boldsymbol{z};q) \nonumber  \\
&=(1-z_3)^{-1} \sum_{\substack{\boldsymbol{k}=(k_1,k_2)  \\ |\boldsymbol{k}| \text{ even}}} 
Q_{\boldsymbol{k}}(z_3;q) z_1^{k_1} z_2^{k_2}+
\sum_{\substack{\boldsymbol{k}=(k_1,k_2)  \\ |\boldsymbol{k}| \text{ odd}}} Q_{\boldsymbol{k}}(z_3;q) z_1^{k_1} z_2^{k_2},
\end{align}
where $|\boldsymbol{k}|:=k_1+k_2$. Since $g$ is invariant under the action generated by
\eqref{simpleaction}, we have the formal functional equations
\begin{equation}  \label{Pfuncinit}
P_{j}(z_1,z_2;q)=(\sqrt{q} z_1)^{j-a_j} P_j  \left( \frac{1}{qz_1},z_2;q \right)
=(\sqrt{q} z_2)^{j-a_j} P_j \left(z_1,\frac{1}{q z_2};q \right), 
\end{equation}
and 
\begin{equation}
Q_{\boldsymbol{k}}(z_3;q)=(\sqrt{q} z_3)^{|\boldsymbol{k}|-a_{|\boldsymbol{k}|}} 
Q_{\boldsymbol{k}} \left( \frac{1}{qz_3}; q \right), \label{Qfuncinit}
\end{equation}
where $a_n=0$ or $1$ according to whether $n$ is even or odd respectively.  

\subsection{Definition of series and some properties} \label{mdsdef}
Suppose $a_1 c_1, a_2 c_2 \in \text{Div}(N)$.
Our arguments will use the following multiple Dirichlet series (formally defined as) 
\begin{equation} \label{multidir}
Z^{(N)}(s,w;\chi_{a_2 c_2},\chi_{a_1 c_1};\pi_f):=\sum_{\substack{m_1,m_2,d \geq 1 \\ \gcd(m_1 m_2 d,2N)=1}}
\frac{H^{\pi_f}(m_1,m_2,d) \chi_{a_2 c_2}(d) \chi_{a_1 c_1}(m_1 m_2)}{(m_1 m_2)^{s} d^{w} },
\end{equation}
where the coefficients $H^{\pi_f}(m_1,m_2,d)$ will now be defined using the recipe of 
\cite[Section~4]{CG1}. 

The function $H^{\pi_f}(m_1,m_2,d)$ on quadruples of odd integers satisfies a twisted multiplicativity 
property. For $(m_1 m_2 d, m_1^{\prime} m_2^{\prime} d^{\prime})=1$ we have 
\begin{equation} \label{twistmult}
H^{\pi_f}(m_1 m_1^{\prime}, m_2 m_2^{\prime}, d d^{\prime})=H^{\pi_f}(m_1,m_2,d) H^{\pi_f}(m_1^{\prime},m_2^{\prime},d^{\prime}) 
\pfrac{d}{m_1^{\prime} m_2^{\prime}}  \pfrac{d^{\prime}}{m_1 m_2}.
\end{equation}
Given property \eqref{twistmult}, it suffices to define $H^{\pi_f}(p^{k_1},p^{k_2},p^{j})$ for all primes $p \nmid 2N$. 
These coefficients are recorded by the generating function (called the $p$-part):
\begin{align*}
Z_p^{(N)}(s,w;\chi_{a_2 c_2},\chi_{a_1 c_1};\pi_{f})&:=\sum_{k_1,k_2,j \geq 0} \frac{H^{\pi_f}(p^{k_1},p^{k_2},p^j) \chi_{a_2 c_2}(p)^{j} 
\chi_{a_1 c_1}(p)^{k_1+k_2}}{p^{(k_1+k_2)s+jw}} \nonumber \\
&=g_{A_3} \left(\chi_{a_1 c_1}(p) \alpha_p p^{-s}, \chi_{a_1 c_1}(p) \beta_p p^{-s}, \chi_{a_2 c_2}(p) p^{-w};p \right),
\end{align*}
where $\alpha_p, \beta_p$ are 
the Satake parameters attached to $\pi_{f,p}$.
In other words,
\begin{equation} \label{primepower}
H^{\pi_f}(p^{k_1},p^{k_2},p^{j}):=a(k_1,k_2,j;p) \alpha_p^{k_1} 
   \beta_p^{k_2} \quad  \text{ for all primes } p \nmid 2N \quad \text{and}  \quad j,k_1,k_2 \geq 0,
\end{equation}
where $a(k_1,k_2,j;p)$ are the coefficients in the power series expansion in \eqref{powerser}.

Using
\begin{equation*}
Z^{(N)}(s,w;\chi_{a_2 c_2},\chi_{a_1 c_1};\pi_f)=\prod_p Z^{(N)}_p(s,w;\chi_{a_2 c_2},\chi_{a_1 c_1};\pi_f),
\end{equation*}
we see that $Z^{(N)}(s,w;\chi_{a_2 c_2},\chi_{a_1 c_1};\pi_f)$ converges absolutely 
for $\Re(s),\Re(w) \gg 1$. 

We also have  
 \begin{align}
H^{\pi_f}(p^{k_1},p^{k_2},p^{j})=  \alpha_p^{k_1} 
  \beta_p^{k_2}, & \quad \text{when}  \quad \min(k_1+k_2,j)=0; \label{fact1} 
 \end{align}
 \begin{equation}
H^{\pi_f}(p^{k_1},p^{k_2},p^{j})=0, \quad \text{when}  \quad \min(k_1+k_2,j)=1; \label{fact2}  
\end{equation}
\begin{equation}
\hspace{1.2cm} H^{\pi_f}(p^{k_1},p^{k_2},p^{j})=0, \quad \text{when}  \quad k_1+k_2 \equiv j \equiv 1 \pmod{2}. \label{fact3}
\end{equation}

For $d,m_1,m_2$ positive integers with $(dm_1m_2,2N)=1$,
consider the following Dirichlet polynomials built from the $P_j$ and 
$Q_{\boldsymbol{k}}$ given in \eqref{Pjdef} and \eqref{Qkdef}:
\begin{align} \label{correct1}
 \mathcal{P}_d(s,\chi_{a_1 c_1};\pi_f)&:=
   \prod_{\substack{p^{j} \parallel d \\ j \geq 2 \\ j \equiv 0 \hspace{0.1cm} (2) } }
  P_j \left(\chi_{a_1 c_1}(p) \chi_{dp^{-j}}(p) \alpha_p p^{-s},
 \chi_{a_1 c_1}(p) \chi_{dp^{-j}}(p) \beta_p p^{-s};p \right) \nonumber \\
 & \times \prod_{\substack{p^{j} \parallel d \\ j \geq 2 \\ j \equiv 1 \hspace{0.1cm} (2) } } 
P_j(\alpha_p p^{-s},\beta_p p^{-s};p); \\
\mathcal{Q}_{m_1,m_2}(w,\chi_{a_2 c_2};\pi_f)&:= 
\prod_{\substack{p^{k_1} \parallel m_1  \\ p^{k_2} \parallel m_2  \\ \boldsymbol{k} \equiv 0 \hspace{0.1cm} (2) \\ |\boldsymbol{k}| \geq 2 } } 
 \alpha_p^{k_1} \beta_p^{k_2}  Q_{\boldsymbol{k}} \big(\chi_{a_2 c_2}(p) \chi_{m_1 m_2 p^{-k_1-k_2}}(p) p^{-w};p \big) \nonumber \\
& \times \prod_{\substack{p^{k_1} \parallel m_1  \\ p^{k_2} \parallel m_2  \\ |\boldsymbol{k}| \equiv 1 \hspace{0.1cm}
 (2) \\ |\boldsymbol{k}| \geq 2} } \label{correct2}
\alpha_p^{k_1} \beta_p^{k_2} Q_{\boldsymbol{k}}(p^{-w};p); \\
\widetilde{\mathcal{Q}}_{n}(w,\chi_{a_2 c_2};\pi_f)&:= \sum_{m_1 m_2=n} 
\mathcal{Q}_{m_1,m_2}(w,\chi_{a_2 c_2};\pi_f). \label{correct3}
\end{align}
Observe that 
\begin{align}
\widetilde{\mathcal{Q}}_n(w,\chi_{a_2 c_2};\pi_f)&:=\prod_{\substack{p^k \parallel n \\
k \equiv 0 \hspace{0.1cm} (2) \\ k \geq 2}}
 \Big( \sum_{\substack{k_1+k_2=k \\ k_1,k_2 \geq 0}} 
 \alpha_p^{k_1} \beta_p^{k_2}  Q_{\boldsymbol{k}} \big(\chi_{a_2 c_2}(p) \chi_{n p^{-k}}(p) p^{-w};p \big) \Big ) \nonumber \\
 & \times \prod_{\substack{p^k \parallel n \\
k \equiv 1 \hspace{0.1cm} (2) \\ k \geq 2}} \Big( \sum_{\substack{k_1+k_2=k \\ k_1,k_2 \geq 0}} 
 \alpha_p^{k_1} \beta_p^{k_2} Q_{\boldsymbol{k}} \big(p^{-w};p \big) \Big ).
\end{align}

The Dirichlet polynomials $\mathcal{P}_d$ (resp. $\widetilde{\mathcal{Q}}_n$)
inherit functional equations from the local ones for $P_j$ (resp. $Q_{\boldsymbol{k}}$) 
given in \eqref{Pfuncinit} (resp. \eqref{Qfuncinit}).

\begin{lemma} \label{corrfunc}
Let $f \in S^{\emph{new}}_{\ell}(N)$ and $a_1 c_1, a_2 c_2 \in \emph{Div}(N)$.
Suppose that $d=d_0 d_1^2$ and $n=n_0 n_1^2$ where $(dn,2N)=1$ and $\mu^2(d_0)=\mu^2(n_0)=1$.
Then 
\begin{equation} \label{Pfunc}
 \mathcal{P}_d(s,\chi_{a_1 c_1};\pi_f)=d_1^{2-4s} \mathcal{P}_d(1-s,\chi_{a_1 c_1};\pi_f),
\end{equation}
and 
 \begin{equation} \label{Qfunc2}
\widetilde{\mathcal{Q}}_{n}(w,\chi_{a_2 c_2};\pi_f)=n_1^{1-2w} \widetilde{\mathcal{Q}}_{n} (1-w,\chi_{a_2 c_2};\pi_f),
 \end{equation} 
 where $\mathcal{P}_d(s,\chi_{a_1 c_1};\pi_f)$ and $\widetilde{\mathcal{Q}}_{n}(w,\chi_{a_2 c_2};\pi_f)$
 are defined by \eqref{correct1} and \eqref{correct3}.
\end{lemma}

The Dirichlet polynomials $\mathcal{P}_d$ and 
 $\widetilde{\mathcal{Q}}_{n}$ satisfy crude bounds.
\begin{lemma} \label{corrbound}
Let $f \in \mathcal{S}^{\emph{new}}_{\ell}(N)$ and $a_1 c_1, a_2 c_2 \in \emph{Div}(N)$. 
Suppose that $d=d_0 d_1^2$ and $n=n_0 n_1^2$ where $(dn,2N)=1$ and $\mu^2(d_0)=\mu^2(n_0)=1$.
Then 
\begin{equation} \label{Pbound}
| \mathcal{P}_d(s,\chi_{a_1 c_1};\pi_f)| \ll_{\varepsilon}
\begin{cases}
d_1^{\varepsilon}, & \Re(s) \geq \frac{1}{2} \\
d_1^{2-4 \Re(s)+\varepsilon}, & \Re(s) < \frac{1}{2},
\end{cases}
\end{equation}
and 
\begin{equation} \label{Qbound} 
|\widetilde{\mathcal{Q}}_{n}(w,\chi_{a_2 c_2};\pi_f)| \ll_{\varepsilon}
\begin{cases}
n_1^{1/2+\varepsilon}, & \Re(w) \geq \frac{1}{2} \\
n_1^{3/2-2 \Re(w)+\varepsilon}, & \Re(w) < \frac{1}{2},
\end{cases}
\end{equation}
where $\mathcal{P}_d(s,\chi_{a_1 c_1};\pi_f)$ and $\widetilde{\mathcal{Q}}_{n}(w,\chi_{a_2 c_2};\pi_f)$
are defined by \eqref{correct1} and \eqref{correct3}.
The implied constants depend only on $\varepsilon>0$.
\end{lemma}
\begin{proof}
This follows from a straightforward modification of the argument in \cite[Appendix~B]{Di}
using the maximum principle and Cauchy's inequality.
\end{proof}

The Dirichlet polynomials $ \mathcal{P}_d(s,\chi_{a_1 c_1};\pi_f)$ and $\widetilde{\mathcal{Q}}_{n}(w,\chi_{a_2 c_2};\pi_f)$
 are in known in the literature as \emph{correction polynomials}. Their purpose 
 is to give two different representations 
of \eqref{multidir}, each absolutely convergent in different tube domains in $\mathbb{C}^2$ \cite{BFH,HK}.
\begin{lemma} \label{rep1lem}
Let $f \in S^{\emph{new}}_{\ell}(N)$ and $a_1 c_1, a_2 c_2 \in \emph{Div}(N)$.
We have 
\begin{equation} \label{rep1}
Z^{(N)}(s,w;\chi_{a_2 c_2}, \chi_{a_1 c_1};\pi_f)=\sum_{\substack{d \geq 1 \\ d=d_0 d_1^2  \\ (d,2N)=1}}
 \frac{L^{(2N)}(s,\pi_f \otimes \chi_{a_1 c_1 d_0}) \chi_{a_2 c_2}(d) \mathcal{P}_d(s,\chi_{a_1 c_1};\pi_f)}{d^w},
\end{equation}
on the domain 
\begin{equation} \label{omega1}
\Omega_1:=\{(s,w) \in \mathbb{C}^2: 2 \Re(s)+\Re(w)>2 \} \cap \{(s,w): \Re(w)>1 \},
\end{equation}
and 
\begin{equation} \label{rep2}
Z^{(N)}(s,w; \chi_{a_2 c_2}, \chi_{a_1 c_1};\pi_f)=
\sum_{\substack{n \geq 1 \\ (n,2N)=1 \\ n=n_0 n_1^2}}
 \frac{L^{(2N)}(w, \chi_{a_2 c_2} \widetilde{\chi}_{n_0}) 
\chi_{a_1 c_1}(n) \widetilde{\mathcal{Q}}_{n}(w,\chi_{a_2 c_2};\pi_f)}{n^s},
\end{equation} 
on the domain 
\begin{equation} \label{omega2}
\Omega_2:=\Big \{(s,w) \in \mathbb{C}^2: \Re(s)+\Re(w)>\frac{3}{2}  \Big \} \cap \{(s,w): \Re(s)>1 \},
\end{equation}
with exception of a polar hyperplane $w=1$ when $a_2=c_2=1$. 
Thus the functions
\begin{equation} \label{holintermed}
(w-1) Z^{(N)}(s,w;\chi_{a_2 c_2}, \chi_{a_1 c_1};\pi_f) 
\end{equation}
are holomorphic on $\Omega_1 \cup \Omega_2$. 
\end{lemma}
\begin{proof}

The formulas \eqref{rep1} and \eqref{rep2} hold for $\Re(s),\Re(w) \gg 1$
by a straightforward modification of the computations in \cite[Section~3]{Di}. 

Then
\eqref{rep1} (resp. \eqref{rep2}) can be extended to hold on the domain \eqref{omega1} 
(resp. \eqref{omega2}) using Heath--Brown's quadratic large sieve \cite[Corollary~3]{HB}
together with \eqref{functionaleq2} and \eqref{Pfunc} (resp. \eqref{dirtwistfunc} and 
\eqref{Qfunc2}), and the bound
\eqref{Pbound} (resp. \eqref{Qbound}). 
 
\end{proof}

\subsection{Functional equations, meromorphic continuation and residues}
Recall that recall $N$ is as in \eqref{factorise} and $f \in S_{\ell}^{\text{new}}(N)$ 
is twist minimal. 

Our argument in Section \ref{mainargument}
requires the exact scattering matrix for the functional equations involving
$Z^{(N)}(s,w;\chi_{a_2 c_2}, \chi_{a_1 c_1};\pi_f)$. Their precise shape 
is determined by the factorisation of $N$ in  \eqref{factorise}
(in other words the ramified local representations $\pi_{f,p}$). 

Following ~\cite[Section~4]{DGH}, we will store the multiple Dirichlet 
series defined above in vector form. 
Denote 
\begin{equation*}
\overrightarrow{\boldsymbol{Z}}^{(N)}(s,w; \chi_{\text{Div}(N)}, \chi_{a_1 c_1};\pi_f), \quad  
\big(\text{resp.} \quad \overrightarrow{\boldsymbol{Z}}^{(N)}(s,w; \chi_{a_2 c_2}, \chi_{\text{Div}(N)};\pi_f) \big)
\end{equation*}
the $4 \mathsf{rad}(N) \times 1$ column vector whose entries are 
\begin{equation*}
Z^{(N)}(s,w; \chi^{(j)}, \chi_{a_1 c_1};\pi_f) \quad (\text{resp.} \quad
Z^{(N)} \big(s,w; \chi_{a_2 c_2},\chi^{(j)} ;\pi_f) \big),
\end{equation*}
where $\chi^{(j)}$ for $j=1,\ldots,4 \mathsf{rad}(N)$ range 
over the characters $\chi_{a_2 c_2}$ (resp. $\chi_{a_1 c_1})$.

Let $\gamma_1,\gamma_2: \mathbb{C}^2 \rightarrow \mathbb{C}^2$ be 
two involutions defined by 
\begin{equation*}
\gamma_1(s,w)=(1-s,w+2s-1) \quad \text{and} \quad \gamma_2(s,w)=(s+w-1/2,1-w).
\end{equation*}
The symmetry group generated by these two involutions is isomorphic to the dihedral
group of order $8$.

A version of the following lemma is an implicit in \cite{HK}. 
\begin{lemma} \label{matrixfunc}
Let $f \in \mathcal{S}^{\text{\emph{new}}}_{\ell}(N)$ be twist minimal and $\Omega_1$ 
\emph{(}resp. $\Omega_2$\emph{)} be as in \eqref{omega1} 
\emph{(}resp. \eqref{omega2}\emph{)} 
respectively. For each $a_1 c_1 \in \emph{Div}(N)$, there exists a 
$4 \mathsf{rad}(N) \times 4 \mathsf{rad}(N)$ matrix $\Phi_{a_1 c_1}(s;\pi_f)$ of meromorphic 
functions in $s$ such that 
for all $(s,w) \in \Omega_1$, we have 
\begin{equation} \label{exacfunc1}
\overrightarrow{\boldsymbol{Z}}^{(N)}(s,w; \chi_{\emph{Div}(N)}, \chi_{a_1 c_1};\pi_f)= 
\Phi_{a_1 c_1}(s;\pi_f) \overrightarrow{\boldsymbol{Z}}^{(N)}(1-s,w+2s-1; \chi_{\emph{Div}(N)}, \chi_{a_1 c_1};\pi_f).
\end{equation}

For each $a_2 c_2 \in \emph{Div}(N)$, there exists a 
$4 \mathsf{rad}(N) \times 4 \mathsf{rad}(N)$ matrix $\Psi_{a_2 c_2}(w;\pi_f)$ of meromorphic 
functions in $w$ such that 
for all $(s,w) \in \Omega_2$ we have 
\begin{equation} \label{exacfunc2}
\overrightarrow{\boldsymbol{Z}}^{(N)}(s,w; \chi_{a_2 c_2}, \chi_{\emph{Div}(N)};\pi_f)= 
\Psi_{a_2 c_2}(w;\pi_f) \overrightarrow{\boldsymbol{Z}}^{(N)}(s+w-1/2,1-w; \chi_{a_2 c_2}, \chi_{\emph{Div}(N)};\pi_f).
\end{equation}
\end{lemma}

For twist minimal $f \in S^{\text{new}}_{\ell}(N)$ (trivial nebentypus), \cite[Proposition~2.8]{LW} implies that
\begin{equation} \label{Nzero}
N_0=\prod_{\substack{p \mid N \\ \pi_{f,p} \text{ special} \\ \text{representation}}} p,
\end{equation}
and 
\begin{equation} \label{None}
N_1=\prod_{\substack{p \mid N \\ \pi_{f,p} \text{ supercuspidal} \\ \text{representation}}} p.
\end{equation}
For each $p \mid N_0$, let $\alpha_p$ be the Satake parameter attached to the special representation 
$\pi_{f,p}$. For these primes we have $\alpha_p^2=p^{-1}$.

For $c_1,c_2,c_2^{\prime} \mid \mathsf{rad}(N)$,
define
\begin{equation} \label{specialprod}
\mathcal{N}(\pi_f)_{c_1 c_2 c_2^{\prime}}:=\prod_{\substack{p \mid N 
 \\ \text{ord}_p \big(c_2 c_2^{\prime} \mathfrak{c} (\pi_f \otimes \chi_{\chi_{-1}(c_1)c_1}) \big) \text{ odd}   }}  \hspace{-0.5cm} p.
\end{equation}
For a Dirchlet character $\chi$, let $\widetilde{\mathfrak{c}}(\chi):=(\mathfrak{c}(\chi),8)$.
 
\begin{lemma} \label{matrixfunc2}
Let $N$ be as in \eqref{factorise}, $f \in \mathcal{S}^{\emph{new}}_{\ell}(N)$ be twist minimal,
and $\mathcal{N}(\pi_f)_{c_1 c_2 c_2^{\prime}}$ 
be as in \eqref{specialprod}. Then for 
$a_1 c_1, a_1^{\prime} c_1^{\prime}, a_2 c_2, a_2^{\prime} c_2^{\prime} \in \emph{Div}(N)$, 
we have the formulae
\begin{align} \label{tildephi}
\Phi&_{a_1 c_1}(s;\pi_f)_{a_2c_2a_2^{\prime}c_2^{\prime}} \nonumber \\
&=2^{-2} \varepsilon(\pi_f \otimes \chi_{\chi_{-1}(c_1) c_1}) 
\cdot \delta_{\mathcal{N}(\pi_f)_{c_1 c_2 c_2^{\prime} }  \mid \frac{N_0}{(c_1,N_0)}}
\cdot \chi_{\chi_{-1}(c_1) a_1} \big({-{\mathfrak{c}(\pi_f \otimes \chi_{\chi_{-1}(c_1)c_1}})} \big) \nonumber \\
& \times \big({\mathfrak{c}(\pi_f \otimes \chi_{\chi_{-1}(c_1)c_1}}) \big)^{1/2-s} \pi^{-1+2s} 
 \frac{\Gamma \Big(\frac{1-s+\frac{\ell-1}{2}}{2} \Big)
 \Gamma \Big(\frac{1-s+\frac{\ell+1}{2}}{2} \Big)}
{\Gamma \Big(\frac{s+\frac{\ell-1}{2}}{2} \Big) \Gamma \Big(\frac{s+\frac{\ell+1}{2}}{2} \Big)} 
\chi_{a_1 c_1} \big(\mathcal{N}(\pi_f)_{c_1 c_2 c_2^{\prime}} \big) \nonumber \\
& \times \Bigg \{ \widetilde{\mathfrak{c}}(\chi_{\chi_{-1}(c_1) a_1})^{1-2s} 
\Bigg[ \frac{L(1-s,\pi_{f,2} \otimes \chi_{a_1 c_1})}{L(s,\pi_{f,2} \otimes \chi_{a_1 c_1})} +\chi_{a_2 a_2^{\prime} }(5)
\frac{L(1-s,\pi_{f,2} \otimes \chi_{5 a_1 c_1 })}{L(s,\pi_{f,2} \otimes \chi_{5 a_1 c_1})} \Bigg] \nonumber \\
& + \widetilde{\mathfrak{c}}(\chi_{\chi_{-1}(c_1) a_1 3  } )^{1-2s} \chi_{\chi_{-1}(\mathfrak{c}(\pi_f \otimes \chi_{\chi_{-1}(c_1) c_1}) 
\cdot \mathcal{N}(\pi_f)_{c_1 c_2 c_2^{\prime}} )}(3) \nonumber  \\
& \times \Bigg [\chi_{a_2 a_2^{\prime}}(3) \frac{L(1-s,\pi_{f,2} \otimes \chi_{3a_1 c_1})}{L(s,\pi_{f,2} \otimes \chi_{3a_1 c_1})}
+\chi_{a_2 a_2^{\prime}}(7) \frac{L(1-s,\pi_{f,2} \otimes \chi_{7a_1 c_1})}{L(s,\pi_{f,2} \otimes \chi_{7a_1 c_1})} \Bigg ]  \Bigg \} \nonumber \\
& \times \prod_{p \mid \mathcal{N}(\pi_f)_{c_1 c_2 c_2^{\prime}} }
\frac{\alpha_p  \left(p^{-(1-s)}-p^{-s} \right)}{1-p^{-3+2s}} \cdot 
\prod_{p \mid \frac{N_0}{(c_1,N_0)}/\mathcal{N}(\pi_f)_{c_1 c_2 c_2^{\prime}}  } 
 \frac{1-p^{-2}}{1-p^{-3+2s}},
\end{align}
and
\begin{align} \label{tildepsi}
\Psi&_{a_2 c_2}(w;\pi_f)_{a_1c_1a_1^{\prime}c_1^{\prime}} \nonumber \\
&=2^{-2} c_2^{\frac{1}{2}-w} \pi^{-\frac{1}{2}+w}
\chi_{a_2 c_2} \Big( \frac{c_1 c_1^{\prime}}{(c_1,c_1^{\prime})^2} \Big) 
\prod_{ p \mid \frac{c_1 c_1^{\prime}}{(c_1,c_1^{\prime})^2}}  \frac{p^{-(1-w)}-p^{-w}}{1-p^{-2(1-w)}}    
\prod_{p \mid \mathsf{rad}(N) / \frac{c_1 c_1^{\prime}}{(c_1,c_1^{\prime})^2 }} 
 \frac{1-p^{-1}}{1-p^{-2(1-w)}} \nonumber  \\
 & \times
 \begin{cases}
\widetilde{\mathfrak{c}}(\chi_{a_2 c_2})^{\frac{1}{2}-w} \frac{\Gamma(\frac{1-w}{2})}{\Gamma(\frac{w}{2})}
\Big [\frac{L_2(1-w,\chi_{a_2 c_2})}{L_2(w,\chi_{a_2 c_2})} 
+ 
\chi_{a_1 a_1^{\prime}}(5)  \frac{L_2(1-w,\chi_{a_2 c_2} \widetilde{\chi}_5)}{L_2(w,\chi_{a_2 c_2} \widetilde{\chi}_5)} \Big] \vspace{0.2cm} \\
+ \widetilde{\mathfrak{c}}(\chi_{-3a_2 c_2})^{\frac{1}{2}-w} \frac{\Gamma(\frac{2-w}{2})}{\Gamma(\frac{1+w}{2})}
\Big [ \chi_{a_1 a_1^{\prime}}(3) \frac{L_2(1-w,\chi_{a_2 c_2} \widetilde{\chi}_3 )}{L_2(w,\chi_{a_2 c_2} \widetilde{\chi}_3)} 
+ 
\chi_{a_1 a_1^{\prime}}(7)  \frac{L_2(1-w,\chi_{a_2 c_2} \widetilde{\chi}_7 )}{L_2(w,\chi_{a_2 c_2} \widetilde{\chi}_7)} \Big] , & a_2 \in \{1,2\} \\
\widetilde{\mathfrak{c}}(\chi_{a_2 c_2})^{\frac{1}{2}-w} \frac{\Gamma(\frac{2-w}{2})}{\Gamma(\frac{w+1}{2})}
\Big [\frac{L_2(1-w,\chi_{a_2 c_2})}{L_2(w,\chi_{a_2 c_2})} 
+ 
\chi_{a_1 a_1^{\prime}}(5)  \frac{L_2(1-w,\chi_{a_2 c_2} \widetilde{\chi}_5)}{L_2(w,\chi_{a_2 c_2} \widetilde{\chi}_5)} \Big] \vspace{0.2cm} \\
+ \widetilde{\mathfrak{c}}(\chi_{-3a_2 c_2})^{\frac{1}{2}-w} \frac{\Gamma(\frac{1-w}{2})}{\Gamma(\frac{w}{2})}
\Big [ \chi_{a_1 a_1^{\prime}}(3) \frac{L_2(1-w,\chi_{a_2 c_2} \widetilde{\chi}_3 )}{L_2(w,\chi_{a_2 c_2} \widetilde{\chi}_3)} 
+ 
\chi_{a_1 a_1^{\prime}}(7)  \frac{L_2(1-w,\chi_{a_2 c_2} \widetilde{\chi}_7 )}{L_2(w,\chi_{a_2 c_2} \widetilde{\chi}_7)} \Big] 
,   & a_2 \in \{-1,-2\}.
 \end{cases}
 \end{align}
\end{lemma}

\begin{proof}
These computations are a generalisation of the ideas in proof of \cite[Theorem~2.3]{DW}.

Fix a full set of squarefree positive representatives for
 $\sfrac{(\frac{\mathbb{Z}}{8 \mathsf{rad}(N) \mathbb{Z}})^{\times}}{(\frac{\mathbb{Z}}{8 \mathsf{rad}(N) \mathbb{Z}})^{\times 2}}$, 
 denote it by $C_{8 \mathsf{rad}(N)}$. 

 In order to use \eqref{functionaleq2}, write
 \begin{equation*}
 \pi_f \otimes \chi_{acd_0}=(\pi_f \otimes \chi_{\chi_{-1}(c)c}) \otimes \chi_{\chi_{-1}(c) ad_0}.
 \end{equation*}
Denote
\begin{align*}
\Lambda_N (s,\pi_f \otimes \chi_{a c d_0})&:= 
\Bigg( \frac{\pi}{\widetilde{\mathfrak{c}}(\chi_{\chi_{-1}(c) ad_0})  \sqrt{\mathfrak{c}(\pi_f \otimes \chi_{\chi_{-1}(c)c}  )}} \Bigg)^{-s}  
 \Gamma \Big(\frac{s+\frac{\ell-1}{2}}{2} \Big) \Gamma \Big(\frac{s+\frac{\ell+1}{2}}{2} \Big) \\
& \times \prod_{\substack{p \mid 2N \\ p \nmid \mathfrak{c}(\chi_{acd_0})  }} L(s, \pi_{f,p} \otimes \chi_{a c d_0}). 
\end{align*}

For $p \mid N_1$ we have $L(s,\pi_{f,p})=1$, and for $p \mid N_0$ 
we have $L(s,\pi_{f,p})=(1-\alpha_p p^{-s})^{-1}$. Thus
\begin{equation} \label{localprod}
\prod_{\substack{p \mid 2 N \\ p \nmid \mathfrak{c}(\chi_{acd_0})}} L(s, \pi_{f,p} \otimes \chi_{a c d_0})=
L(s,\pi_{f,2} \otimes \chi_{acd_0})
\cdot \prod_{\substack{p \mid N_0 \\ p \nmid \mathfrak{c}(\chi_{acd_0}) }} \frac{1}{1-\alpha_p \chi_{acd_0}(p)  p^{-s}}.
\end{equation}

For $D \in C_{8 \mathsf{rad}(N)}$ and $(s,w) \in \Omega_1$, the linear combination 
\begin{equation*}
2^{-\omega( \mathsf{rad}(N))-2} \chi_{a_2 c_2}(D) 
 \sum_{a_2^{\prime} c_2^{\prime} \in \text{Div}(N)} 
\chi_{a_2^{\prime} c_2^{\prime}}(D) Z^{(N)}(s,w; \chi_{a_2^{\prime} c_2^{\prime}}, \chi_{a_1 c_1};\pi_f)
\end{equation*} 
isolates summands the summands of $Z^{(N)}(s,w; \chi_{a_2 c_2}, \chi_{a_1 c_1};\pi_f)$ that have 
$d D $ congruent to a square modulo $8 \mathsf{rad}(N)$. We now combine \eqref{functionaleq2}
and Lemma \ref{corrfunc} to obtain the following observation. For $(s,w) \in \Omega_1$,
the function 
\begin{equation} \label{uniform1}
\Lambda_N(s,\pi_f \otimes \chi_{a_1 c_1 D}) 
\sum_{a_2^{\prime} c_2^{\prime} \in \text{Div}(N)} 
\chi_{a_2^{\prime} c_2^{\prime}}(D)  Z^{(N)}(s,w; \chi_{a_2^{\prime} c_2^{\prime}}, \chi_{a_1 c_1};\pi_f), 
\end{equation}
is invariant under the involution $\gamma_1$ (note that $\gamma_1(\Omega_1)=\Omega_1$),
exactly up to the root number (cf. \eqref{twistroot2})
\begin{equation} \label{rootnum}
\varepsilon(\pi_f \otimes \chi_{\chi_{-1}(c_1)  c_1}) 
\chi_{\chi_{-1}(c_1) a_1 D} \big({-\mathfrak{c}(\pi_f \otimes \chi_{\chi_{-1}(c_1)c_1}} ) \big).
\end{equation}

For $(s,w) \in \Omega_1$, write
\begin{align} \label{rewrite}
Z^{(N)}(s,w; \chi_{a_2 c_2}, \chi_{a_1 c_1};\pi_f) 
&=2^{-\omega(\mathsf{rad}(N))-2} \sum_{D \in C_{8 \mathsf{rad}(N)}} \chi_{a_2 c_2}(D) \nonumber \\
& \times \sum_{a_2^{\prime} c_2^{\prime} \in \text{Div}(N) } \chi_{a_2^{\prime} c_2^{\prime}}(D) 
Z^{(N)}(s,w;\chi_{a_2^{\prime} c_2^{\prime}}, \chi_{a_1 c_1};\pi_f).
\end{align}

Using \eqref{uniform1}, \eqref{rootnum} and \eqref{rewrite}, 
we see that the following holds for all $(s,w) \in \Omega_1$,
\begin{align}  \label{intermedeq}
Z^{(N)}(s,w; \chi_{a_2 c_2},\chi_{a_1 c_1};\pi_f)
&=2^{-\omega(\mathsf{rad}(N))-2} \varepsilon(\pi_f \otimes \chi_{\chi_{-1}(c_1) c_1}) \chi_{\chi_{-1}(c_1) a_1} 
\big({-\mathfrak{c}(\pi_f \otimes \chi_{\chi_{-1}(c_1)c_1}})\big) \nonumber \\
& \times \sum_{a_2^{\prime} c_2^{\prime} \in \text{Div}(N)} 
Z^{(N)}(1-s,w+2s-1;\chi_{a_2^{\prime} c_2^{\prime}}, \chi_{a_1 c_1};\pi_f)  \nonumber \\
& \times \sum_{D \in C_{8 \mathsf{rad}(N)}} \chi_{a_2 c_2}(D) \chi_{a_2^{\prime} c_2^{\prime}}(D) 
 \chi_{D} \big({-\mathfrak{c}(\pi_f \otimes \chi_{\chi_{-1}(c_1)c_1}}) \big) \nonumber \\
& \times \frac{\Lambda_N(1-s,\pi_f \otimes \chi_{a_1 c_1 D})}{\Lambda_N(s,\pi_f \otimes \chi_{a_1 c_1 D})}.
\end{align}

Thus by \eqref{localprod},
\begin{align}
\frac{\Lambda_N(1-s,\pi_f \otimes \chi_{a_1 c_1 D})}{\Lambda_N(s,\pi_f \otimes \chi_{a_1 c_1 D})}
&=\pi^{-1+2s} \cdot \widetilde{\mathfrak{c}} \big( \chi_{\chi_{-1}(c_1) a_1 D} \big)^{1-2s} 
\cdot \big(\mathfrak{c} \big(\pi_f \otimes \chi_{\chi_{-1}(c_1)c_1}) \big)^{1/2-s} \nonumber  \\ 
& \times \frac{\Gamma \Big(\frac{1-s+\frac{\ell-1}{2}}{2} \Big) \Gamma \Big(\frac{1-s+\frac{\ell+1}{2}}{2} \Big)}
{\Gamma \Big(\frac{s+\frac{\ell-1}{2}}{2} \Big) \Gamma \Big(\frac{s+\frac{\ell+1}{2}}{2} \Big)} 
\cdot \frac{L(1-s,\pi_{f,2} \otimes \chi_{a_1 c_1 D})}{L(s,\pi_{f,2} \otimes \chi_{a_1 c_1 D})}
\nonumber \\ 
& \times \prod_{\substack{p \mid N_0 \\ p \nmid c_1 }} 
\frac{1-p^{-2}+\alpha_p \chi_{a_1 c_1 D}(p)(p^{-(1-s)}-p^{-s} ) }{1- p^{-3+2s}}. \label{product}
\end{align}

After substitution of \eqref{product} into \eqref{intermedeq}, and expansion of the products, we see that 
each term will vanish when the summation over $D$ is performed unless the total character in $D$
is trivial. This is only possible for summands corresponding to $c_2^{\prime} \mid \mathsf{rad}(N)$ that
satisfy 
\begin{equation*}
\mathcal{N}(\pi_f)_{c_1 c_2 c_2^{\prime} }  \mid \frac{N_0}{(c_1,N_0)}.
\end{equation*}
Thus we obtain \eqref{tildephi}.

Note that \eqref{tildepsi} follows from an analogous, but simpler computation.
\end{proof}

\begin{prop} \label{merocont}
Let $f \in \mathcal{S}^{\emph{new}}_{\ell}(N)$ be twist minimal.
For each $a_1 c_1, a_2 c_2 \in \emph{Div}(N)$, the function
\begin{equation} \label{holofunc}
w(w-1)(2s+w-1)(2s+w-2) Z^{(N)}(s,w;\chi_{a_2 c_2}, \chi_{a_1 c_1};\pi_f)
\end{equation}
has a holomorphic extension to all of $\mathbb{C}^2$. For each
$(z,w) \in \mathbb{C}^2$, there exists a constant 
$C:=C \big ({\Re(z),\Re(w),\ell} \big)$ such that 
we have
\begin{align} \label{modintermedpolar}
\big | w (w-1)(2s+w-1)(&2s+w-2) Z^{(N)}(s,w;\chi_{a_2 c_2}, \chi_{a_1 c_1};\pi_f) \big | \nonumber \\
& \ll_{\pi_f,c_1,c_2,\Re(z),\Re(w)} \Big[ \big(1+|\Im(s)| \big) \big(1+|\Im(w)| \big) \Big]^{C}.   
\end{align}
\end{prop}
\begin{proof}
This a straightforward modification of the 
several complex variable arguments in \cite[Lemma~2]{Blo} and 
\cite[Section~4.3]{DGH}. 

Both \eqref{tildephi} and \eqref{tildepsi} are both holomorphic functions 
for $\Re(s)<1$ and $\Re(w)<1$ respectively. 

One iteratively applies \eqref{exacfunc1} and \eqref{exacfunc2}
to holomorphically extend each of the functions 
in \eqref{holofunc} to all of $\mathbb{C}^2 \setminus P^{\star}$, 
where
\begin{equation*}
P^{\star}:=\{(s,w): (\Re s, \Re w) \in P \} \subseteq \{(s,w): |\Re(s)|^2+|\Re(w)|^2 \leq 10 \} ,
\end{equation*}
and $P \subseteq \mathbb{R}^2$ is a certain closed polygon.
Bochner's Theorem \cite{Bo} then holomorphically extends the functions in 
\eqref{holofunc} to all of $\mathbb{C}^2$. One can then use \cite[Propositions~4.6 and 4.7]{DGH}
and the argument on \cite[pg.~341]{DGH} to establish \eqref{modintermedpolar}.
\end{proof}

\begin{remark}
We point out that the bound \eqref{modintermedpolar} is not used in 
obtaining the main results of this
paper. Its purpose it to ensure absolute convergence for certain contour integrals in $s$ and $w$
after Mellin inversion of smooth functions. 
\end{remark}

\begin{lemma} \label{R1}
Let $f \in \mathcal{S}^{\emph{new}}_{\ell}(N)$,
$a_1 c_1, a_2 c_2 \in \emph{Div}(N)$, and consider 
$Z^{(N)}(s,w;\chi_{a_2 c_2}, \chi_{a_1 c_1};\pi_f)$
on the domain $(s,w) \in \Omega_2$ given in \eqref{omega2}.
If $\chi_{a_2 c_2}$ is non-trivial, then the function  
$Z^{(N)}(s,w;\chi_{a_2 c_2}, \chi_{a_1 c_1};\pi_f)$ is holomorphic on $\Omega_2$.
If $a_2=c_2=1$, then there is a polar hyperplane at $w=1$ with residue
\begin{equation*} 
\emph{Res}_{w=1} Z^{(N)}(s,w;\mathbf{1},\chi_{a_1 c_1};\pi_f)=
L^{(2N)}(2s,\emph{Sym}^2 \pi_f ) \prod_{p \mid 2N}(1-p^{-1}).
\end{equation*}
\end{lemma}
\begin{proof}
Recall the representation \eqref{rep2}, valid for $(s,w) \in \Omega_2$.
 The only summands 
of \eqref{rep2} that have poles are those with $a_2=c_2=n_0=1$, and these are each simple and 
come from the Dirichlet $L$--function. 

Then 
\begin{align} \label{resid1}
\text{Res}_{w=1} & Z^{(N)}(s,w;\mathbf{1}, \chi_{a_1 c_1};\pi_f) \\
 &=\prod_{p \mid 2N} (1-p^{-1}) \sum_{\substack{m_1,m_2 \geq 1 \\ m_1 m_2=\square \\ 
(m_1 m_2, 2N)=1}} \frac{\mathcal{Q}_{m_1,m_2}(1;\mathbf{1};\pi_f)}{(m_1 m_2)^s} \nonumber \\
&=\prod_{p \mid 2N} (1-p^{-1}) \prod_{p \nmid 2N}  \Bigg(1+ \sum_{\substack{k_1+k_2 \geq 2 \\ k_1+k_2 \text{ even}} } 
\frac{\alpha_p^{k_1} \beta_p^{k_2} Q_{k_1,k_2}(p^{-1};p)}{p^{(k_1+k_2)s}} \Bigg) \nonumber  \\
&=\prod_{p \mid 2N} (1-p^{-1}) \prod_{p \nmid 2N} (1-p^{-1}) g_3^{+}(\alpha_p p^{-s},\beta_p p^{-s},p^{-1};p) \nonumber \\
&=\prod_{p \mid 2N} (1-p^{-1}) \prod_{p \nmid 2N} (1-p^{-2s})^{-1} (1-\alpha_p^2 p^{-2s} )^{-1} (1-\beta_p^2 p^{-2s} )^{-1}, \nonumber \\
&=L^{(2N)}(2s,\text{Sym}^2 \pi_f) \prod_{p \mid 2N} (1-p^{-1}).
\end{align}
Note that the third to last display follows from \eqref{Qkdef}, and the penultimate display from direct computation using
\eqref{plusminus} and \eqref{gdef}. 
\end{proof}

\section{Voronoi and Poisson summation with combinatorial weights} \label{mainargument}
Let $W: (0,\infty) \rightarrow \mathbb{R}$ be a smooth function compactly supported on $[1,2]$.
For a given twist minimal $f \in S^{\text{new}}_{\ell}(N)$, 
we
asymptotically evaluate 
\begin{equation} \label{moment}
M_{\pi_f}(X):=\sum_{\substack{d=d_0 d_1^2 \\ \mu^2(d_0)=1 \\ (d,2N)=1}} 
W \Big( \frac{d}{X} \Big) \frac{1}{d^{1/2}}
L^{(2N)}(1/2, \pi_f \otimes \chi_{d_0} ) \mathcal{P}_d (1/2;\mathbf{1};\pi_f), \quad X \rightarrow \infty,
\end{equation}
with error term uniform in both $X$ and $N$. 

The conductor of each 
$L (1/2, \pi_f \otimes \chi_{d_0} )$ in \eqref{moment} is $\asymp_{\ell} X^2 N$.
Let $R$ be a parameter satisfying $1 \leq R \ll_{\ell} X^2 N$. 
For $\varepsilon>0$ fixed and small, let $U:=(XN)^{\varepsilon}$
and $\gamma$ denote the straight line 
segment $[\varepsilon-iU,\varepsilon+iU]$.
For $a_1^{\prime} c_1^{\prime},  a_2^{\prime} c_2^{\prime} \in \text{Div}(N)$, 
consider the expressions
\begin{align}  
S&_{\pi_f}(X,R,a_1^{\prime} c_1^{\prime}) \nonumber \\
&:=\frac{1}{{c^{\prime}_1}^{1/2} } \int_{\gamma} 
\int_{\gamma} 
 \Bigg |\sum_{\substack{1 \leq d \leq UR c_1^{\prime}/X \\ (d,2N)=1 }} 
\frac{\mathcal{P}_{d}(1/2+s-w,\chi_{a_1^{\prime} c_1^{\prime}};\pi_f)}{d^{1/2+w}} \nonumber  \\
& \times \sum_{\substack{1 \leq n \leq U R   \\  (n,2N)=1 }} 
 \frac{\lambda_f(n) \chi_{a_1^{\prime} c_1^{\prime}}(n) \chi_{d_0}(n)}{n^{1/2+s-w}} \Bigg | 
  \Big |dw \frac{ds}{s}     \Big |; \label{S1} 
 \end{align}
 and 
\begin{align}
\widetilde{S}&_{\pi_f}(X,R,a_1^{\prime} c_1^{\prime},a_2^{\prime} c_2^{\prime} ) \nonumber \\
&:=\delta_{(c_1^{\prime},c_2^{\prime})=1} \cdot \delta_{c_2^{\prime} \mid N_0}
 \cdot \frac{1}{{c_1^{\prime}}^{1/2}} \frac{c_2^{\prime}}{N_0}  \int_{\gamma} \int_{\gamma} \nonumber \\
& \Bigg | \sum_{\substack{1 \leq d \leq
UXN N_0 c_1^{\prime}/R \\ (d,2N)=1  }} 
\frac{\mathcal{P}_{d}(1/2+w-s;\chi_{a^{\prime}_1 c^{\prime}_1} ; \pi_f) 
\chi_{a_2^{\prime} c_2^{\prime}}(d)}{d^{1/2+2s-w}}  \nonumber \\
& \times \sum_{\substack{1 \leq n \leq U X^2 N N_0/(R c_2^{\prime}) \\ (n,2N)=1}} 
\frac{\lambda_f(n) \chi_{a^{\prime}_1 c^{\prime}_1}(n) \chi_{d_0}(n)}{n^{1/2+w-s}}
\Bigg | \Big | \frac{ds}{s} dw     \Big |. \label{S2}
\end{align}

\begin{prop} \label{poissonprop}
Let $N$ be as in \eqref{factorise}, $f \in \mathcal{S}^{\emph{new}}_{\ell}(N)$ be twist minimal,
and $M_{\pi_f}$ 
be as in \eqref{moment}. 
Suppose that the statements in \eqref{GL2lindelof} and \eqref{GL3lindelof} hold with exponents 
$0<\delta_1,\delta_2<1/4$ respectively.

Then for $\varepsilon>0$, $X \geq 1$,  and $1 \leq R \ll_{\ell} X^2 N$,
\begin{align} \label{Mpif}
M_{\pi_f}(X)&=X^{1/2} \cdot \big(1+\delta_{N=\square} \cdot 
\varepsilon(\pi_f) \big) \cdot \widehat{W}(1/2) \cdot L^{(2N)}(1,\emph{Sym}^2 \pi_f)  
\cdot \prod_{p \mid 2N} (1-p^{-1})  \nonumber \\
&+ O_{\ell,\varepsilon} \Bigg( U \sum_{a^{\prime}_1 c^{\prime}_1 \in \emph{Div}(N)  } 
S_{\pi_f}(X,R,a_1^{\prime} c_1^{\prime} ) \Bigg) \nonumber \\
&+O_{\ell,\varepsilon} \Bigg( U \sum_{a_1^{\prime} c_1^{\prime}, a_2^{\prime} c_2^{\prime} \in \emph{Div}(N)}
   \widetilde{S}_{\pi_f}(X,R,a_1^{\prime} c_1^{\prime},a_2^{\prime} c_2^{\prime} ) \Bigg)  \nonumber \\
&+O_{\ell,\varepsilon} \Bigg( U \Big( 
\frac{(N_0^2 N_1^3)^{1/4-\delta_2} X^{1/2}}{R^{1/4}} 
+\frac{(N_0^2 N_1^3)^{1/4-\delta_2} R^{1/4}}{N^{1/4} N_0^{3/4}} +1 \Big) \Bigg),
\end{align}
where $U=(XN)^{\varepsilon}$.
\end{prop}

\begin{remark}
Consider the extremal case when $\mathfrak{c}(\pi_f)=N$ is a perfect square. 
When one twists $\pi_f$ by an a Dirichlet character $\chi_{d_0}$ 
with $d_0>0$ (i.e. it is even) and $(d_0,2N)=1$, 
then $\varepsilon(\pi_f \otimes \chi_{d_0})=\varepsilon(\pi_f)$.  

In this case, we would expect the main term to double
when $\varepsilon(\pi_f)=1$, and vanish identically when $\varepsilon(\pi_f)=-1$.
This explains the factor $1+\delta_{N=\square} \cdot \varepsilon(\pi_f)$ in 
Proposition \ref{poissonprop}. 
 \end{remark}

\begin{proof}
Let $H$ an even holomorphic function that satisfies $H(0)=1$ and
the growth estimate 
\begin{equation*}
|H(z)| \ll_{\Re(z),C} (1+|z|)^{-C}, \quad \text{for any} \quad C>0.
\end{equation*}

Consider the integral
\begin{equation} \label{Jpir}
J_{\pi_f}(X,R):=\frac{1}{(2 \pi i)^2} \int_{(4.6)} \int_{(2)} \widehat{W}(w) H(s)
 Z^{(N)} (1/2+s,1/2+w;\mathbf{1},\mathbf{1};\pi_f) X^{w} R^{s} \frac{ds}{s} dw.
\end{equation}
We move the $s$-contour to $\Re(s)=-2$ and encounter a pole at $s=0$. The residue is  
\begin{equation*}
\frac{1}{2 \pi i} \int_{(4.6)} \widehat{W}(w) Z^{(N)}(1/2,1/2+w;\mathbf{1},\mathbf{1};\pi_f) X^{w} dw = M_{\pi_f}(X),
\end{equation*}
where the equality follows from Lemma \ref{rep1lem} and Mellin inversion. 

We make the change of variable 
$s \rightarrow -s$ in the shifted integral, 
and then apply \eqref{exacfunc1}. The net result is 
\begin{equation} \label{mastereqn}
M_{\pi_f}(X)=J_{\pi_f}(X,R) +
\sum_{a_2^{\prime} c_2^{\prime} \in \text{Div}(N)} 
 K_{\pi_f}(X,R,a^{\prime}_2 c^{\prime}_2),
\end{equation}
where 
\begin{align} 
K_{\pi_f}(X,R,a^{\prime}_2 c^{\prime}_2)&:=\frac{1}{(2 \pi i)^2} \int_{(4.6)} \int_{(2)} \widehat{W}(w) H(s) 
 \Phi_{11}(1/2-s;\pi_f)_{11a_2^{\prime} c_2^{\prime}} \nonumber  \\
& \times Z^{(N)} (1/2+s,1/2+w-2s;\chi_{a_2^{\prime} c_2^{\prime}} ,\mathbf{1};\pi_f) X^{w} R^{-s} 
\frac{ds}{s} dw.  \label{Kpiac}
\end{align} 

\subsection{Treatment of $J_{\pi_f}(X,R)$} \label{Jpisec}
We interchange the $s$ and $w$ integrations in \eqref{Jpir} by Fubini's Theorem, 
and then move the $w$-contour to $\Re(w)=-1$. We encounter a pole at $w=1/2$, 
whose residue can be computed using Lemma \ref{R1}. No other poles are encountered
because the domain of integration remains in $\Omega_2$ when this perturbation of contour is made 
(see Lemmas \ref{rep1lem} and \ref{R1}).

We make the change of variable $w \rightarrow -w$ in the shifted integral, 
and then apply \eqref{exacfunc2}. The net result is
\begin{align} \label{Jpi}
J_{\pi_f}(X,R)&=\frac{1}{2 \pi i}  \widehat{W}(1/2) X^{1/2} 
\prod_{p \mid 2N} (1-p^{-1})
 \int_{(2)}H(s) L^{(2N)}(1+2s,\text{Sym}^2 \pi_f) R^{s} \frac{ds}{s} \nonumber \\
&+ \sum_{\substack{a_1^{\prime} c_1^{\prime} \in \text{Div}(N) }}
 I_{\pi_f}(X,R,a^{\prime}_1 c^{\prime}_1),
\end{align}
where 
\begin{align*}
I_{\pi_f}(X,R,a^{\prime}_1 c^{\prime}_1)&:=\frac{1}{(2 \pi i)^2} \int_{(2)} \int_{(1)} 
\widehat{W}(-w) H(s) \Psi_{11}(1/2-w;\pi_f)_{11a_1^{\prime}  c_1^{\prime}}  \nonumber  \\
& \times Z^{(N)}(1/2+s-w,1/2+w;\mathbf{1},\chi_{a_1^{\prime} c_1^{\prime}};\pi_f) 
X^{-w} R^{s} dw \frac{ds}{s}.
\end{align*}

\subsubsection{First main term}
We move the $s$-contour in the first term of \eqref{Jpi} to $\Re(s)=-1/4$, 
encountering a pole at $s=0$.
We obtain 
\begin{align} \label{first}
\frac{1}{2 \pi i} &  \widehat{W}(1/2) X^{1/2} \prod_{p \mid 2N} (1-p^{-1}) \int_{(2)}H(s) 
L^{(2N)}(1+2s,\text{Sym}^2 \pi_f) R^{s} \frac{ds}{s} \nonumber \\
&=X^{1/2}  \widehat{W}(1/2) L^{(2N)}(1,\text{Sym}^2 \pi_f) \prod_{p \mid 2N} (1-p^{-1}) 
+O_{\ell,\varepsilon} \big( U (N_0^2 N_1^3)^{1/4-\delta_2} X^{1/2} R^{-1/4} \big),
\end{align}
where the error term follows from hypothesis \eqref{GL3lindelof} and a trivial estimation
of the missing Euler factors.
Thus \eqref{first} gives one of the main terms and one of the error terms 
in Proposition \ref{poissonprop}.
\subsubsection{Treatment of $I_{\pi_f}(X,R,a^{\prime}_1 c^{\prime}_1)$}
Using Lemma \ref{rep1lem} we obtain
\begin{align} \label{Ipiint}
I_{\pi_f}(X,R,a^{\prime}_1 c^{\prime}_1)&:=\frac{1}{(2 \pi i)^2}   
 \sum_{\substack{d \geq 1 \\ (d,2N)=1}}  \frac{1}{d^{1/2}}
\sum_{\substack{n \geq 1 \\ (n,2N)=1}} \frac{\lambda_f(n) \chi_{a_1^{\prime} c_1^{\prime}}(n) 
 \chi_{d_0}(n) }{n^{1/2}}  \nonumber \\  
 & \times \int_{(2)} \int_{(1)} \widehat{W}(-w) H(s) 
  \Psi_{11}(1/2-w;\pi_f)_{11a_1^{\prime} c_1^{\prime}} \nonumber  \\
 & \times \mathcal{P}_d(1/2+s-w,\chi_{a_1^{\prime} c_1^{\prime}};\pi_f)  
 \left( \frac{dX}{n} \right)^{-w} \left( \frac{n}{R} \right)^{-s}     
 dw  \frac{ds}{s}.
\end{align}

Observe from \eqref{tildepsi} that
\begin{equation} \label{psisimp}
\Psi_{a_2^{\prime} c_2^{\prime}}(1/2-w;\pi_f)_{11a_1^{\prime} c_1^{\prime}}=
\delta_{(c_1^{\prime},c_2^{\prime})=1}
{c_2^{\prime}}^{w} {c_1^{\prime}}^{-1/2+w} A(w), \qquad \Re(w) \geq \varepsilon,
\end{equation}
where $A(w)$ is a holomorphic in the given half-plane. 
For each $w \in \mathbb{C}$ with $\Re(w) \geq \varepsilon$, 
there exists a constant $C_1:=C_1 \big({\Re(w)} \big)$
such that 
\begin{equation} \label{Abound}
|A(w)| \ll_{\Re(w),\varepsilon} 
N^{\varepsilon} (1+|\Im(w)|)^{C_1}.
\end{equation}

We move the $s$--contour in \eqref{Ipiint} to $\Re(s)=B_1+2$ 
for sufficiently large $B_1>0$.
Recall the divisor bound $|\lambda_f(n)| \leq d(n)$ \cite{De} 
and Lemma \ref{corrbound} for the weights $\mathcal{P}_d$.
The net contribution 
to $I_{\pi_f}(X,R,a^{\prime}_1 c^{\prime}_1)$ from all $n \geq U R$ is 
$O_{\ell, B_1 \varepsilon} \big((XN)^{-B_2} \big)$
and some $B_2>0$.
We truncate the $n$ sum \eqref{Ipiint} to the range $1 \leq n \leq U R$.

We next move the $w$-contour to $\Re(w)=\Re(s)=B_1+2$. 
We again use divisor bounds, as well as \eqref{psisimp} and \eqref{Abound}. 
We truncate 
the $d$-sum to the range $1 \leq d \leq  UR c_1^{\prime}/X$,
incurring a negligible error.

We now move the $s$ and $w$ contours to $\Re(s)=\Re(w)=\varepsilon$.
By the rapid decay of $\widehat{W}$ and $H$ we can truncate
 the $s,w$-integrations in \eqref{Ipiint}
to $|\Im(s)|, |\Im(w)| \leq U$ with negligible error.
We interchange the finite summations with the absolutely convergent integrals,
and use \eqref{psisimp} and \eqref{Abound} in \eqref{Ipiint}. 
We obtain the error terms involving $S_{\pi_f}(X,R,a_1^{\prime} c_1^{\prime})$ in 
Proposition of \ref{poissonprop}.

\subsection{Treatment of $K_{\pi_f}(X,R,a_2^{\prime} c^{\prime}_2)$}
\subsubsection{Second main term}
We interchange the $s$ and $w$ integrations in \eqref{Kpiac} by Fubini's Theorem,
and then move the $w$-contour to $\Re(w)=4$ for each \eqref{Kpiac}. Observe
that the domain of integration remains in $\Omega_2$ 
(see Lemmas \ref{rep1lem} and \ref{R1}). 

We encounter the polar hyperplane $1/2+w-2s=1$ when $a_2^{\prime} c_2^{\prime}=11$,
otherwise there are no poles encountered for this move when 
$a_2^{\prime} c_2^{\prime} \neq 11$. We can 
 compute its residue using Lemma \ref{R1}. 
Observe that \eqref{tildephi} gives 
\begin{equation*}
\Phi_{11}(1/2;\pi_{f})_{1111}=\delta_{N=\square} \cdot \varepsilon(\pi_f).
\end{equation*}
The residue is
\begin{align} \label{k1res}
\frac{1}{2 \pi i} & X^{1/2} \prod_{p \mid 2N} (1-p^{-1}) 
 \int_{(2)} \widehat{W}(1/2+2s) H(s) \Phi_{11}(1/2-s;\pi_f)_{1111} \nonumber \\
& \times L^{(2N)}(1+2s, \text{Sym}^2 \pi_f) X^{2s} R^{-s} \frac{ds}{s} \nonumber \\
&= \delta_{N=\square} \cdot \varepsilon(\pi_f) X^{1/2} L^{(2N)}(1,\text{Sym}^2 \pi_f) \widehat{W}(1/2) 
\prod_{p \mid 2N} (1-p^{-1}) \nonumber \\
&+O_{\ell,\varepsilon} \big( U (N_0^2 N_1^3)^{1/4-\delta_2} N^{-1/4} N_0^{-3/4} R^{1/4} \big),
\end{align}
where the error term follows from \eqref{tildephi} and the hypothesis \eqref{GL3lindelof}.
Substituting \eqref{k1res} back into \eqref{mastereqn} gives the 
second main term and another error term in Proposition \ref{poissonprop}. 

For each $a_2^{\prime} c_2^{\prime} \in \text{Div}(N)$ we are left to estimate 
\begin{align} \label{remainint}
L_{\pi_f}(X,R,a_2^{\prime} c_2^{\prime}):= \frac{1}{(2 \pi i)^2} & \int_{(2)}  \int_{(4)}  \widehat{W}(w) H(s)  
Z^{(N)} (1/2+s,1/2+w-2s;\chi_{a_2^{\prime} c_2^{\prime}},\mathbf{1};\pi_f) \nonumber \\
& \times \Phi_{11}(1/2-s;\pi_f)_{11a_2^{\prime} c_2^{\prime}}  X^{w} R^{-s} dw \frac{ds}{s}.
\end{align}
We handle this in the next argument.

\subsubsection{Remaining terms}
Applying \eqref{exacfunc2} to \eqref{remainint} gives
\begin{align} \label{Kpidevelop}  
L_{\pi_f}(X,R,a^{\prime}_2 c^{\prime}_2)&:=\frac{1}{(2 \pi i)^2} \sum_{a_1^{\prime} c_1^{\prime} \in \text{Div}(N)}
 \int_{(2)}  \int_{(4)} \widehat{W}(w) H(s)  \Phi_{11}(1/2-s;\pi_f)_{11a_2^{\prime} c_2^{\prime}} \nonumber \\ 
& \times \Psi_{a_2^{\prime} c_2^{\prime}}(1/2+w-2s;\pi_f)_{11 a_1^{\prime} c_1^{\prime}} \nonumber  \\
 & \times Z^{(N)} (1/2+w-s,1/2+2s-w;\chi_{a_2^{\prime} c_2^{\prime}} ,\chi_{a_1^{\prime} c_1^{\prime}};\pi_f)  \nonumber \\
& \times X^{w} R^{-s} dw \frac{ds}{s}.
\end{align}

We will again use the decay of $\Phi$ coming from the $\pi_{f,p}$ that
are special representations. 
A particular case of the formula \eqref{tildephi} can be written as
\begin{equation} \label{phidecay} 
\Phi_{11}(1/2-s;\pi_f)_{11a_2^{\prime} c_2^{\prime}}= 
N^{s} \delta_{c_2^{\prime} \mid N_0} 
\Big( \frac{N_0}{c_2^{\prime}} \Big)^{-1+s} B(s), \qquad \Re(s) \geq \varepsilon,
\end{equation}
where $B(s)$ is a holomorphic in the given half plane. 
For each $s \in \mathbb{C}$ with $\Re(s) \geq \varepsilon$, 
there exists a constant $C_2:=C_2 \big({\Re(s),\ell} \big)$
such that
\begin{equation} \label{Bbound}
|B(s)| \ll_{\Re(s),\ell,\varepsilon} 
N^{\varepsilon} (1+|\Im(s)|)^{C_2}.
\end{equation}

We use Lemma \ref{rep1lem} to open the multiple Dirichlet series in \eqref{Kpidevelop}.
We then argue similarly to Section \ref{Jpisec}, except now we appeal to
\eqref{psisimp}, \eqref{Abound}, \eqref{phidecay} and \eqref{Bbound}. 
We obtain the error terms $\widetilde{S}_{\pi_f}(X,R,a_1^{\prime} c_1^{\prime}, a_2^{\prime} c_2^{\prime} )$ in 
Proposition \ref{poissonprop}. 
This completes the proof.
\end{proof} 
\section{Endgame} \label{endgame}
\begin{proof}[Proof of Theorem \ref{mainthm}]
We use Proposition \ref{poissonprop} with $X \geq 1$
and  $1 \leq R \ll_{\ell} X^2N$ to be chosen later.

We then apply Lemma \ref{lindelof} to estimate the $S_{\pi_f}$ and 
$\widetilde{S}_{\pi_f}$ terms defined in \eqref{S1} and \eqref{S2} respectively.
We also exploit that the fact that $c_1^{\prime} \mid \mathsf{rad}(N)=N_0 N_1$
in both \eqref{S1} and \eqref{S2}.
The net result is
\begin{align} 
M_{\pi_f}(X)&=X^{1/2} \cdot \big(1+\delta_{N=\square} \cdot \varepsilon(\pi_f) \big) \cdot \widehat{W}(1/2) \cdot L^{(2N)}(1,\text{Sym}^2 \pi_f)  
\cdot \prod_{p \mid 2N} (1-p^{-1})  \nonumber \\
&+ O_{\ell,\varepsilon} \Bigg( U \bigg[  \Big(\frac{R}{X} \Big)^{1-2 \delta_1} (N N_0^3 N_1^3)^{1/4-\delta_1}
+\Big(\frac{X N N_0}{R} \Big)^{1-2 \delta_1} (N N_0^3 N_1^3)^{1/4-\delta_1} \bigg] \Bigg) \nonumber \\
&+O_{\ell,\varepsilon} 
\Bigg( U \bigg[\frac{(N_0^2 N_1^3)^{1/4-\delta_2} X^{1/2}}{R^{1/4}} 
+\frac{(N_0^2 N_1^3)^{1/4-\delta_2} R^{1/4}}{N^{1/4} N_0^{3/4}} +1 \bigg] \Bigg). \nonumber 
\end{align}

We choose $R:=X(NN_0)^{1/2}$ to balance the error 
terms in the middle display above. This yields 
\begin{align} \label{major}
M_{\pi_f}(X)&=X^{1/2} \cdot \big(1+\delta_{N=\square} \cdot \varepsilon(\pi_f) \big) \cdot \widehat{W}(1/2) \cdot L^{(2N)}(1,\text{Sym}^2 \pi_f)  
\cdot \prod_{p \mid 2N} (1-p^{-1})  \nonumber \\
&+O_{\ell,\varepsilon} \bigg( U \Big( N^{9/8-(7/2) \delta_1} N_0^{7/8-(5/2) \delta_1}
+\frac{X^{1/4} N^{1/4-(3/2) \delta_2}}{N_0^{\delta_2/2}}
 \Big)   \bigg).
\end{align}

 A Theorem of Hoffstein and Lockhart \cite[pg.~164]{HL} ensures that
\begin{equation*}
L(1,\text{Sym}^2 \pi_f) \gg_{\ell,\varepsilon} N^{-\varepsilon},
\end{equation*}
with ineffective constant depending on $\varepsilon>0$. Taking 
\begin{equation*}
X \gg_{\varepsilon} N^{\varepsilon} \Big( N^{9/4-7 \delta_1} N_0^{7/4-5 \delta_1}
+\frac{N^{1-6 \delta_2  }}{N_0^{2 \delta_2}} \Big)
\end{equation*}
shows that the error terms in \eqref{major} are smaller 
than the main term for $M_{\pi_f}(X)$.
Hence there is 
 a fundamental discriminant $d$ in the desired range such that
 \begin{equation*}
L^{(2N)}(1/2, \pi_f \otimes \chi_d) \neq 0.
\end{equation*}
A trivial estimation of the missing Euler factors 
yields Theorem \ref{mainthm}.
\end{proof}
\bibliographystyle{amsalpha}
\bibliography{AJDTwist4}
\end{document}